\documentclass[twoside,12pt]{article}
\usepackage{amsmath,amssymb,amsthm,amsfonts}
\usepackage{paralist}
\usepackage{subfigure}
\usepackage[title]{appendix}
 
\usepackage{graphics} 
\usepackage{epsfig} 
\usepackage{graphicx}
\usepackage{caption}
\usepackage{epstopdf}
\usepackage[colorlinks=true]{hyperref}
\usepackage{float}
\usepackage{pdfpages}
\usepackage{amsfonts}
\usepackage{amsmath,amsthm}
\usepackage{multirow}
\usepackage{cancel}
\usepackage{amsmath}
\usepackage{mathrsfs}
\usepackage{txfonts}
\usepackage{enumerate}
\usepackage[numbers,sort&compress]{natbib}
\usepackage[font=scriptsize,labelfont=bf,width=0.95\textwidth]{caption}

\numberwithin{equation}{section}
\newcommand{\T}{\mathbb T}
\newcommand{\R}{\mathbb R}
\newcommand{\A}{\mathcal A}
\newcommand{\J}{\mathcal J}
\newcommand{\K}{\mathcal K}

\newcommand{\qc}{q_c}
\newcommand{\gp}{\gamma'}

\newcommand{\be}{\begin{equation}}
\newcommand{\ee}{\end{equation}}

\newcommand{\ben}{\begin{eqnarray*}}
\newcommand{\enn}{\end{eqnarray*}}

\newtheorem{definition}{Definition}
\newtheorem{proposition}{Proposition}[section]
\newtheorem{theorem}{\textbf Theorem}[section]
\newtheorem{lemma}{\textbf Lemma}[section]

 \numberwithin{equation}{section}
\newtheorem{remark}{Remark}[section]
\usepackage{microtype}
\title{Endpoint Maximal Regularity for Superquadratic Hamilton–Jacobi Equations and Applications to Ergodic Mean-field Games Systems}

 \author{ Fanze Kong \thanks{Department of Applied Mathematics, University of Washington, Seattle, WA 98195, USA; {\sf Email: fzkong@uw.edu}} }
\date{}

\begin{document}

\maketitle

 \begin{abstract}
A celebrated conjecture of P.-L. Lions concerns maximal regularity for viscous Hamilton--Jacobi equations. In this paper, we study the endpoint case. We consider normalized strong solutions of 
$$-\Delta u+|Du|^\gamma=f$$ in $\mathbb T^d$, where $d\geq 2$, $\gamma>2$, and $f\in L^{q_c}(\mathbb T^d)$ with $q_c=d(\gamma-1)/\gamma$. At this critical exponent, the main difficulty is possible concentration under the critical scaling. Assuming that the source terms form a uniformly equi-integrable subset of $L^{q_c}$, we rule out this concentration and prove maximal $L^{q_c}$ regularity for strong solutions. The proof combines a two-stage blow-up argument with a Liouville rigidity theorem. More precisely, the second blow-up yields a uniform local \(L^{\gamma q_c}\)-bound for the gradients, while the small drift arising from the first blow-up upgrades weak convergence to strong local compactness, ultimately leading to a contradiction with Liouville rigidity.
Finally, we apply the maximal regularity theory for Hamilton--Jacobi equations at the endpoint case to establish the existence of ergodic solutions to defocusing second-order mean-field games systems with critical coupling exponents.

\end{abstract}

\section{Introduction}

We consider the periodic viscous Hamilton--Jacobi equation
\begin{equation}\label{eq:intro-HJ}
 -\Delta u+|Du|^\gamma=f
 \qquad \text{in }\mathbb T^d,
\end{equation}
where $d\geq 2$, $\gamma>1$ and $f$ is a given source term.   A fundamental question, raised by P.-L.~Lions in a series of seminars on
viscous Hamilton--Jacobi equations, is whether the nonlinear term
$|Du|^\gamma$ inherits the same $L^q$ integrability as the source $f$.  More precisely, Lions conjectured that, whenever
\begin{align}\label{q_c_def}
 q>q_c:=\frac{d(\gamma-1)}{\gamma},
\end{align}
equation \eqref{eq:intro-HJ} enjoys the so-called maximal
$L^q$-regularity
\begin{equation}\label{eq:intro-MR}
 \|D^2u\|_{L^q(\mathbb T^d)}
 +\bigl\||Du|^\gamma\bigr\|_{L^q(\mathbb T^d)}
 \leq C,
\end{equation}
where $C$ depends only on $d$, $\gamma$ and  the source term $f$.  Here, the exponent $q_c$ is the scaling-critical exponent
associated with \eqref{eq:intro-HJ}.  In fact, if $ u_r(x):=r^{\frac{2-\gamma}{\gamma-1}}u(rx)$, then both $-\Delta u$ and $|Du|^\gamma$ have the same homogeneity, while
the $L^{q_c}$ norm of the rescaled right-hand side is invariant, which implies $q=q_c$ is critical from the viewpoint of compactness. 
The maximal $L^q$ estimates are useful, for instance, in providing compactness for the
existence theory of mean-field games systems and for vanishing-viscosity
approximations of first-order Hamilton--Jacobi equations;  see, e.g.,
\cite{CrandallLions1983,LasryLions2007,LionsLecture2009,LionsLecture2012,GomesPimentelSanchezMorgado2015,Munoz2022}.    The study of the regularity and rigidity of (\ref{eq:intro-HJ}) with superquadratic Hamiltonians can be traced to the Bernstein method  developed by Lions \cite{Lions1985}; see also
\cite{PeletierSerrin1978} for related Liouville theorems. 
H\"older regularity was established by Dall'Aglio and
Porretta \cite{DallAglioPorretta2015}; some H\"older estimates for
coercive Hamiltonians were obtained in \cite{CapuzzoDolcettaLeoniPorretta2010}. 

The Lions' conjecture has been confirmed in several important regimes. In the superquadratic Hamiltonians ($\gamma>2$) case, Dall'Aglio and Porretta
\cite{DallAglioPorretta2015} proved H\"older regularity for weak solutions.
Cirant and Verzini \cite{CirantVerzini2022} established local maximal
$L^q$-regularity for $q>q_c$ and $\gamma>2$.  In the periodic setting, Cirant and Goffi
\cite{CirantGoffi2021} obtained global maximal $L^q$ estimates for
$q>q_c$ and $\gamma>2$.  More recently, in a subquadratic regime, Cirant et al.
obtained local $W^{2,p}$ estimates for Hamilton--Jacobi equations by using weighted Morrey-space estimates \cite{CirantKongWeiZeng2025}.  
There are also some results for the maximal regularity of the parabolic Hamilton--Jacobi equations.  For superquadratic Hamiltonians, Cardaliaguet and Silvestre
\cite{CardaliaguetSilvestre2012} established local H\"older estimates for parabolic Hamilton--Jacobi equations with unbounded right-hand sides. Cirant and Goffi obtained Lipschitz estimates \cite{CirantGoffiLipschitz2020}, while
maximal regularity above the critical exponent was established in
\cite{CirantGoffiParabolic2021}. In the superquadratic regime, maximal regularity throughout the full range
strictly above the critical exponent was subsequently established via a
blow-up argument in \cite{Cirant2025}. In particular, the relevant  Liouville results for the parabolic Hamilton--Jacobi equation can be traced to
\cite{PolacikQuittnerSouplet2007}. The endpoint case $q=q_c$, however, is delicate:   as observed in \cite{CirantVerzini2022},
boundedness of the source term in $L^{q_c}$ does not prevent concentration at
arbitrarily small scales.  Additional compactness assumptions are therefore
required to recover maximal regularity at the endpoint.   We remark that the endpoint maximal-regularity problem exhibits distinct
behavior in the subquadratic and superquadratic regimes. Concerning the subquadratic Hamiltonians, Goffi
\cite{Goffi2023} established endpoint maximal regularity under uniform
equi-integrability assumptions on the source term.   The endpoint argument in \cite{Goffi2023} relies on a duality-based
stability estimate in the space
$L^{d(\gamma-1)/(2-\gamma)}$, which is valid only when $\gamma<2$.  Then, in the superquadratic regime, Cirant, Goffi  and Leonori
\cite[Theorem~3.1(ii) and Remark~3.2]{CirantGoffiLeonori2025} established the maximal regularity via the integral Bernstein method  under the
condition $\gamma>\frac{d}{d-2},$
assuming a uniform equi-integrability of
the source terms.   In view of Remark 2.14 in
\cite{CirantGoffiLeonori2025}, the maximal regularity at the endpoint case also holds when $d=2$. Consequently, the range not covered by their framework is $d=3$, $2<\gamma\leq3.$

The purpose of this paper is to develop a unified framework for establishing endpoint maximal \(L^{q_c}\)-regularity for \eqref{eq:intro-HJ} throughout the superquadratic regime \(\gamma>2\). In particular, our result closes the remaining gap mentioned above.  Beyond this extension, our proof develops a different approach based on a two-stage blow-up argument and Liouville rigidity, instead of the integral Bernstein method. 

Noting that $ W^{2,q_c}(\mathbb T^d)
 \hookrightarrow C^{0,\alpha_c}(\mathbb T^d)$, where
\begin{align}\label{alphac}
\alpha_c:= 2-\frac{d}{q_c}
 =
 2-\frac{\gamma}{\gamma-1}
 =
 \frac{\gamma-2}{\gamma-1}
 \in(0,1),
\end{align}
which is insufficient
for maximal $L^q$ regularity since the embedding is not compact  even though the solution belongs to $C^{0,\alpha_c}$.   We assume that the set of admissible source terms in (\ref{eq:intro-HJ}), denoted by
$\mathcal F$, is bounded and uniformly equi-integrable in
$L^{q_c}(\mathbb T^d)$.  More precisely, a set
$\mathcal F\subset L^{q_c}(\mathbb T^d)$ is said to be bounded and uniformly
equi-integrable if $\sup_{f\in\mathcal F}\|f\|_{L^{q_c}(\mathbb T^d)}<\infty$
and its equi-integrability modulus
\begin{align}\label{eq:UInt}
 \omega_{\mathcal F}(s)
 :=
 \sup_{f\in\mathcal F}
 \sup_{\substack{E\subset\mathbb T^d\  \\ |E|\le s}}
 \int_E |f|^{q_c}\,dx\rightarrow 0  \qquad\text{as }s\downarrow0.
\end{align}
Under the above assumptions, we establish the following endpoint maximal regularity result for the Hamilton--Jacobi equation.

\begin{theorem}\label{thm:main}
Let $d\ge2$, $\gamma>2$ and
$q_c=d(\gamma-1)/\gamma$.  Suppose that
$\mathcal F\subset L^{q_c}(\mathbb T^d)$ is bounded and uniformly
equi-integrable. Then, there exists a constant $ C=C\bigl(d,\gamma,
 \sup_{f\in\mathcal F}\|f\|_{L^{q_c}(\mathbb T^d)},
 \omega_{\mathcal F}\bigr)$
such that  for $f\in\mathcal F$, every strong solution
$u\in W^{2,q_c}(\mathbb T^d)$ of \eqref{eq:intro-HJ}, with
  $\int_{\mathbb T^d}u\,dx=0$, 
satisfies
\begin{align}\label{eq:main-estimate}
 \|D^2u\|_{L^{q_c}(\mathbb T^d)}
 +\bigl\||Du|^\gamma\bigr\|_{L^{q_c}(\mathbb T^d)}
 \le C.
\end{align}
\end{theorem}
\begin{remark}
Under the assumption that the source terms are uniformly equi-integrable,
Theorem~\ref{thm:main} gives an affirmative answer to the superquadratic
endpoint problem raised by Cirant and Verzini
\cite{CirantVerzini2022}, establishing maximal \(L^{q_c}\)-regularity for
\eqref{eq:intro-HJ}. Compared with
\cite[Theorem~3.1(ii)]{CirantGoffiLeonori2025}, the   new exponent
range is $d=3$, $2<\gamma\leq3$.  Moreover, our proof follows a different approach, based on a two-stage
blow-up argument and Liouville rigidity, and provides a unified proof
for all \(d\geq2\) and \(\gamma>2\).
\end{remark}
The novelty of the proof is a two-stage blow-up argument recover  the compactness at the critical regime of the blow-up sequence under the 
uniformly equi-integrability assumption of the source term, where some partial regularity induced by the small drift is significant.  In the framework of the first blow-up for showing uniformly H\"older vanishing property, we have to establish the compactness of the blow-up sequence. To this end, we perform the second blow-up to obtain the local energy estimate  and use the small-drift regularity result to upgrade the weak compactness to the strong compactness, which ensures the the crucial uniformly H\"older vanishing property of the solution by deriving the contradiction with Liouville rigidity theorem. 
Finally, the endpoint
Gagliardo--Nirenberg inequality and a standard Calder\'on--Zygmund estimate then
complete the maximal-regularity estimate given in Theorem \ref{thm:main}.

Our second set of main results concerns applications of
Theorem~\ref{thm:main}. More precisely,  we apply the maximal-regularity result  to establish the existence of solutions to ergodic defocusing mean-field games systems.
We consider the following ergodic mean-field games systems:
\begin{equation}\label{eq:mfg-system}
 \begin{cases}
 -\Delta u+H(Du)+\lambda=m^{\bar\alpha}+V(x),&\text{ in }\T^d,\\
 -\Delta m-\nabla\cdot\!\left(mD_pH(Du)\right)=0, &\text{ in }\T^d\\
 m\ge0,\quad \displaystyle\int_{\T^d}m\,d x=1,
 \quad \displaystyle\int_{\T^d}u\,d x=0,
 \end{cases}
\end{equation}
where \((u,m,\lambda)\) is the unknown solution triple, \(H(p)=|p|^\gamma\), \(D_pH(p)=\gamma|p|^{\gamma-2}p\) and $\gamma>2$.  Here we suppose that $V\geq 0$ and $V\in C^2(\R^2).$  

The system \eqref{eq:mfg-system} arises in the study of the long-time behavior of mean-field games. Mean-field games theory was introduced independently by Lasry and Lions \cite{LasryLions2007} and by Huang, Malham\'e and Caines \cite{HuangMalhameCaines2006} to characterize Nash equilibria in stochastic differential games with a large population of indistinguishable agents.  When the horizon is infinite, by using the ergodicity, one finds the equilibrium is described by a stationary system coupling a Hamilton--Jacobi
equation with a Fokker--Planck equation.  In this formulation,
$u$ denotes the value function of a representative agent, $m$ is the
population density of agents, and $\lambda$ is the ergodic constant, representing the
optimal average cost per unit time.     Ergodic second-order mean-field games systems have been investigated by
 fixed-point arguments and variational methods; see, among
others,
\cite{GomesPiresSanchezMorgado2012,GomesPatriziVoskanyan2014,MeszarosSilva2015,MeszarosSilva2018}. In particular, the variational study of ergodic mean-field games systems has been developed
systematically in
\cite{MeszarosSilva2015,MeszarosSilva2018,CesaroniCirant2019,
CirantKongWeiZeng2025,KongTongZengZhou2026,KongTongZeng2026,KongTongZeng2026MountainPass}.
We remark that the qualitative properties of solutions, particularly
uniqueness, depend strongly on the monotonicity of the coupling. For a
monotone increasing coupling, corresponding to the defocusing regime, the
Lasry--Lions monotonicity condition typically yields uniqueness. For a monotone decreasing coupling, corresponding to the focusing regime, uniqueness
cannot generally be expected
 and the problems become delicate.
Stationary focusing problems and their mass-subcritical, mass-critical or Sobolev-critical
regimes have been studied in
\cite{CesaroniCirant2019GroundStates,BernardiniCesaroni2023,CirantCosenzaVerzini2024,
CirantKongWeiZeng2025,KongTongZengZhou2026}.

In this paper, we are concerned with the defocusing mean-field games systems.  In detail, we define $\bar\alpha_*=\frac{\gamma'}{d-2-\gamma'}.$  Choosing
 $\alpha=\bar\alpha_*$ in (\ref{eq:mfg-system}), by using the variational approach, Theorem \ref{thm:main} and Sobolev embedding yield precisely $m^\alpha\in L^{q_c}(\mathbb T^d)$, which is the scaling-critical exponent in the endpoint maximal regularity of Hamilton--Jacobi equations (\ref{eq:intro-HJ}).  We emphasize that this endpoint mechanism differs from the mass-critical
phenomenon arising for focusing mean-field games systems. In the latter case, the energy functional is non-convex. In the present
setting, by contrast, the variational functional remains convex, and the
main obstruction arises instead from the loss of compactness at the endpoint
regularity threshold for the Hamilton--Jacobi equation.  To overcome this difficulty, we first construct a minimizer of the variational problem.  The resulting strong compactness of the
densities yields the uniform equi-integrability of the corresponding
coupling terms, which allows us to apply the endpoint maximal-regularity
estimate to the Hamilton--Jacobi equation in \eqref{eq:mfg-system}.
A linearization of the variational problem, combined with convex duality,
then identifies the minimizing flux with the optimal feedback $w=-mD_pH(Du).$
This closes the Fokker--Planck equation in \eqref{eq:mfg-system} and proves the existence of an
ergodic solution to the second-order defocusing mean-field games system at the endpoint case. The relevant results are summarized as follows.
\begin{theorem} 
\label{thm:main2}
Suppose that $V\in C^2(\T^d)$, $d\geq 4$ and $\bar \alpha=\frac{\gp}{d-2-\gp}$, system
\eqref{eq:mfg-system} admits a unique   variational solution  satisfying
Definition~\ref{def:solution}.  More
precisely, there exist $(u,\lambda,m,w)$ satisfying $m^\alpha\in L^{\qc}(\T^d)$, $D_pH(Du)\in L^d(\T^d;\R^d)$, $w\in L^{r_*}(\T^d;\R^d)$ with $r_*:=\frac{d}{d-1-\gp}>1$ and $q_c$ given in Theorem \ref{thm:main}.
\end{theorem}

The paper is organized as follows.  In Section~\ref{sec:preliminaries} we
collect the preliminaries including Gagliardo--Nirenberg inequality, small-drift compactness of the linearized Hamilton--Jacobi equation and Liouville rigidity result.
Section~\ref{Section3} contains the two-stage blow-up argument for showing the uniformly H\"older  vanishing estimate.  In
Section~\ref{sec:maximal-regularity} we combine this crucial estimate with the
endpoint interpolation Gagliardo--Nirenberg inequality to prove
Theorem~\ref{thm:main}.  Section~\ref{section5} applies Theorem~\ref{thm:main} to establish the existence of solutions to defocusing mean-field games systems with critical coupling exponents.

\section{Preliminary results}\label{sec:preliminaries}
In this section, we introduce the notation and collect several estimates that will be used throughout the paper. Let $\gamma'$ denote the H\"older conjugate exponent of $\gamma$.  Then
\begin{equation}\label{eq:exponents}
 \gamma'=\frac{\gamma}{\gamma-1},
 \qquad
 q_c=\frac{d}{\gamma'}=\frac{d(\gamma-1)}{\gamma}.
\end{equation}
Hence,  $1<q_c<d$ and  $\gamma'q_c=d$.  Moreover, the Sobolev conjugate exponent of $q_c$ is
\begin{equation}\label{eq:critical-identities}
 q_c^*
 :=
 \frac{dq_c}{d-q_c}
 =
 \gamma q_c.
\end{equation}
Let $\mathcal U$ be an arbitrary family of strong solutions of
\eqref{eq:intro-HJ} with zero average, whose corresponding source terms
belong to $\mathcal F$.  Define the H\"{o}lder semi-norm as  
 \[
 [u]_{C^{0,\alpha}(\mathbb T^d)}
 :=
 \sup_{\substack{x,y\in\mathbb T^d\\x\ne y}}
 \frac{|u(x)-u(y)|}
 {\operatorname{dist}(x,y)^\alpha}.
\]

Now, we  give the counterexample of the maximal regularity to (\ref{eq:intro-HJ}) with the absence of the uniformly equi-integrability assumption.
\begin{proposition}
\label{prop:critical-counterexample}
There exist smooth periodic solutions to
\[
 -\Delta u_\varepsilon+|Du_\varepsilon|^\gamma=f_\varepsilon
 \qquad\text{in }\mathbb T^d,
 \qquad
 \int_{\mathbb T^d}u_\varepsilon\,dx=0,
\]
such that
\[
 \sup_{0<\varepsilon<\varepsilon_0}
 \|f_\varepsilon\|_{L^{q_c}(\mathbb T^d)}<\infty,
\]
whereas
\begin{align}\label{divergence}
 \|D^2u_\varepsilon\|_{L^{q_c}(\mathbb T^d)}
 +\bigl\||Du_\varepsilon|^\gamma\bigr\|_{L^{q_c}(\mathbb T^d)}
 \rightarrow\infty
 \qquad\text{as }\varepsilon\downarrow0.
\end{align}
Moreover,    $(f_\varepsilon)_\varepsilon$ is not uniformly
equi-integrable in $L^{q_c}(\mathbb T^d)$.
\end{proposition}

\begin{proof}
Define
\[
 \alpha:=\alpha_c=\frac{\gamma-2}{\gamma-1},
 \qquad
 c:=\frac{(d-2+\alpha)^{1/(\gamma-1)}}{\alpha}.
\]
For $x\in\mathbb R^d\setminus\{0\}$, let $U(x):=cr^\alpha$
with $r:=|x|$, then we have
\[
 |DU|=c\alpha r^{\alpha-1},
 \qquad
 \Delta U=c\alpha(d-2+\alpha)r^{\alpha-2}.
\]
Since  $\alpha-2=-\gamma'$, $ \gamma(\alpha-1)=-\gamma'$,
and $c\alpha(d-2+\alpha)=(c\alpha)^\gamma$, we have
\begin{equation}\label{eq:punctured-profile}
 -\Delta U+|DU|^\gamma=0
 \qquad\text{in }\mathbb R^d\setminus\{0\}.
\end{equation}

Fix $R\in(0,1/8)$ and choose a smooth radial cutoff $\bar\chi$ such that
$\bar\chi=0$ on $[0,1]$ and $\bar\chi=1$ on $[2,\infty)$, and set
\[
 V(y):=c\,\bar\chi(|y|)|y|^\alpha,
 \qquad
 F(y):=-\Delta V(y)+|DV(y)|^\gamma.
\]
Then $V\in C^\infty(\mathbb R^d)$, $V=0$ on $B_1$, and
$V(y)=c|y|^\alpha$ for $|y|\ge2$.  Moreover, $F\not\equiv0$.  For $0<\varepsilon<R/4$, define the regularized profile
\[
 v_\varepsilon(x)
 :=
 \varepsilon^\alpha V(x/\varepsilon)
\]
 and set $\widetilde u_\varepsilon(x)
 :=
 \chi(x)v_\varepsilon(x).$
After periodic extension to $\mathbb T^d$, define the normalized function
\[
 u_\varepsilon
 :=
 \widetilde u_\varepsilon
 -\fint_{\mathbb T^d}\widetilde u_\varepsilon\,dx,
\]
which satisfies
\[
 f_\varepsilon
 :=
 -\Delta u_\varepsilon+|Du_\varepsilon|^\gamma.
\]
Noting that $\chi\equiv1$ on $B_R$ and
$V(y)=c|y|^\alpha$ for $|y|\ge2$, we find

\begin{equation}\label{eq:profile-annulus}
 u_\varepsilon(x)=c|x|^\alpha-C_\varepsilon,
 \qquad\text{whenever }2\varepsilon\le|x|\le R.
\end{equation}
where
$C_\varepsilon:=\fint_{\mathbb T^d}\widetilde u_\varepsilon\,dx$.
In particular, $Du_\varepsilon=DU,$ $D^2u_\varepsilon=D^2U$
on $B_R\setminus B_{2\varepsilon}.$

We next prove the uniform $L^{q_c}$ bound for the source term $f_{\varepsilon}$.  On
$B_{2\varepsilon}$, we have
\[
 -\Delta v_\varepsilon+|Dv_\varepsilon|^\gamma
 =
 \varepsilon^{\alpha-2}F(x/\varepsilon)
 =
 \varepsilon^{-\gamma'}F(x/\varepsilon).
\]
Since $\gamma'q_c=d$,  we find 
\[
 \int_{B_{2\varepsilon}}|f_\varepsilon|^{q_c}\,dx
 =
 \int_{B_2}|F(y)|^{q_c}\,dy.
\]
On $B_R\setminus B_{2\varepsilon}$, \eqref{eq:profile-annulus} and
\eqref{eq:punctured-profile} yield $f_\varepsilon=0.$
On the fixed annulus $B_{2R}\setminus B_R$, both
$\widetilde u_\varepsilon=\chi c|x|^\alpha$ and $f_\varepsilon$ are
independent of $\varepsilon$.  Finally,
$f_\varepsilon=0$ on $\mathbb T^d\setminus B_{2R}$.  Therefore,
\[
 \sup_{0<\varepsilon<R/4}
 \|f_\varepsilon\|_{L^{q_c}(\mathbb T^d)}<\infty.
\]

It remains to prove (\ref{divergence}).
On the annulus $2\varepsilon<|x|<R$, we have
$Du_\varepsilon=DU$ and $D^2u_{\varepsilon}=D^2u$.  
There is a
constant $a_0=a_0(d,\gamma)>0$ such that
\[
 |D^2U(x)|\ge a_0|x|^{\alpha-2}.
\]
Since $q_c(\alpha-2)=-d,$
we obtain
\[
 \begin{aligned}
 \|D^2u_\varepsilon\|_{L^{q_c}(\mathbb T^d)}^{q_c}
 &\ge
 a_0^{q_c}|\mathbb S^{d-1}|
 \int_{2\varepsilon}^R
 r^{q_c(\alpha-2)+d-1}\,dr \\
 &=
 a_0^{q_c}|\mathbb S^{d-1}|
 \log\frac{R}{2\varepsilon},
 \end{aligned}
\]
where \(|\mathbb S^{d-1}|\) denotes the surface measure of the unit sphere.
Similarly, since $\gamma q_c(\alpha-1)=-d,$
we have
\[
 \begin{aligned}
 \bigl\||Du_\varepsilon|^\gamma\bigr\|_{L^{q_c}(\mathbb T^d)}^{q_c}
 &\ge
 (c\alpha)^{\gamma q_c}|\mathbb S^{d-1}|
 \int_{2\varepsilon}^R
 r^{\gamma q_c(\alpha-1)+d-1}\,dr \\
 &=
 (c\alpha)^{\gamma q_c}|\mathbb S^{d-1}|
 \log\frac{R}{2\varepsilon}.
 \end{aligned}
\]
Two terms shown above diverge as $\varepsilon\downarrow0$.

Finally, since $F\not\equiv0$, one has
\[
 \int_{B_{2\varepsilon}}|f_\varepsilon|^{q_c}\,dx
 =
 \int_{B_2}|F|^{q_c}\,dy>0,
 \qquad
 |B_{2\varepsilon}|\rightarrow0.
\]
Thus, $(f_\varepsilon)_\varepsilon$ is not uniformly equi-integrable in
$L^{q_c}(\mathbb T^d)$.
\end{proof}

Proposition \ref{prop:critical-counterexample} implies the   boundedness in $L^{q_c}$ alone is insufficient and we can construct the counterexample, which does not  satisfy 
\eqref{eq:UInt} such that the maximal regularity does not hold; some related disccusions are shown in \cite{CirantVerzini2022}.

Next, we shall give the preliminaries for showing the maximal regularity under the equi-integrability assumption on the source term. To begin with, we   recall the H\"older estimate for \eqref{eq:intro-HJ}, proved by
Dall'Aglio and Porretta \cite{DallAglioPorretta2015}, in the simplified
periodic setting considered here. 

\begin{theorem}[Lemma
2.1 in \cite{CirantVerzini2022}] \label{thm:critical-holder}
For every $M>0$, there is a constant $K=K(d,\gamma,M)$ such that every 
strong solution with zero average to \eqref{eq:intro-HJ} with
$\|f\|_{L^{q_c}}\le M$ satisfies
\begin{equation}\label{eq:global-holder}
 [u]_{C^{0,\alpha_c}(\mathbb T^d)}\le K,
\end{equation}
where $q_c$ and $\alpha_c$ are given in (\ref{q_c_def}) and (\ref{alphac}), respectively.
\end{theorem}


In addition, we have the following H\"{o}lder semi-norm vanishing estimate for $u\in W^{2,q_c}(\mathbb T^d)$.

\begin{lemma}\label{lem:individual-little}
Let $d\ge2$, $\gamma>2$.
Then every $u\in W^{2,q_c}(\mathbb T^d)$ satisfies
\begin{equation}\label{eq:individual-little}
 \lim_{r\downarrow0}
 \sup_{\substack{x,y\in\mathbb T^d\\
 0<\text{dist}(x,y)\le r}}
 \frac{|u(x)-u(y)|}
 {\text{dist}(x,y)^{\alpha_c}}
 =0,
\end{equation}
where $q_c$ and $\alpha_c$ are given in (\ref{q_c_def}) and (\ref{alphac}), respectively.
\end{lemma}

\begin{proof}
Since $\gamma>2$, we have $ q_c=\frac{d(\gamma-1)}{\gamma}>\frac d2$,  Then, by the Morrey embedding, $ W^{2,q_c}(\mathbb T^d)
 \hookrightarrow C^{0,\alpha_c}(\mathbb T^d)$
continuously.

Let $\varepsilon>0$.  Since $C^\infty(\mathbb T^d)$ is dense in
$W^{2,q_c}(\mathbb T^d)$, there exists
$\phi\in C^\infty(\mathbb T^d)$ such that
\[
 [u-\phi]_{C^{0,\alpha_c}(\mathbb T^d)}
 <\frac{\varepsilon}{2}.
\]
Since $\phi$ is smooth, $D\phi\in L^\infty(\mathbb T^d)$.  Choose
$r_\varepsilon>0$ such that
\[
 \|D\phi\|_{L^\infty(\mathbb T^d)}
 r_\varepsilon^{1-\alpha_c}
 <\frac{\varepsilon}{2}.
\]
Then, for every $x,y\in\mathbb T^d$ such that
$0<\operatorname{dist}(x,y)\le r_\varepsilon$, we have
\begin{align*}
 \frac{|u(x)-u(y)|}
 {\operatorname{dist}(x,y)^{\alpha_c}}
 &\le
 [u-\phi]_{C^{0,\alpha_c}(\mathbb T^d)}
 +\frac{|\phi(x)-\phi(y)|}
 {\operatorname{dist} (x,y)^{\alpha_c}} \\
 &\le
 \frac{\varepsilon}{2}
 +\|D\phi\|_{L^\infty(\mathbb T^d)}
 \operatorname{dist}  (x,y)^{1-\alpha_c} <\varepsilon.
\end{align*}
Taking the supremum over such $x$ and $y$ proves
\eqref{eq:individual-little}.
\end{proof}
 
Next, we give the Gagliardo--Nirenberg interpolation inequality.
\begin{lemma}
\label{lem:critical-GN}
Let $B_r(x_0)\subset\mathbb R^d$ be any ball and let
$v\in W^{2,q_c}(B_r(x_0))\cap C^{0,\alpha_c}(B_r(x_0))$.  Then
\begin{equation}\label{eq:critical-GN}
 \|Dv\|_{L^{\gamma q_c}(B_r(x_0))}^\gamma
 \le C[v]_{C^{0,\alpha_c}(B_r(x_0))}^{\gamma-1}
       \|D^2v\|_{L^{q_c}(B_r(x_0))}
 +C[v]_{C^{0,\alpha_c}(B_r(x_0))}^\gamma,
\end{equation}
 
where the constant $C$ depends only on $d$ and $\gamma$.
\end{lemma}

\begin{proof}
 First of all, for $B_r(x_0)=B_1,$ Gagliardo--Nirenberg interpolation inequality
\cite{Nirenberg1959,Triebel1983} gives
\begin{equation}\label{eq:GN-base}
 \|D \psi\|_{L^{q_c^*}(B_1)}
 \le
 C\|D^2 \psi\|_{L^{q_c}(B_1)}^{a}
 [\psi]_{C^{0,\alpha_c}(B_1)}^{1-a}
 +C[\psi]_{C^{0,\alpha_c}(B_1)},
\end{equation}
where $a\in(0,1)$ is determined by
\[
 \frac{1}{q_c^*}-\frac{1}{d}
 =
 a\left(\frac{1}{q_c}-\frac{2}{d}\right)
 -(1-a)\frac{\alpha_c}{d},
\]
and $q^*_c$ is given in \eqref{eq:critical-identities}.  Then, we apply \eqref{eq:GN-base} to $ \psi(y):=r^{-\alpha_c}
 \bigl(v(x_0+ry)-v(x_0)\bigr),$ $y\in B_1$ and obtain from $\alpha_c=2-\frac{d}{q_c}$ and $q_c^*=\gamma q_c$ that 
\[
 [\psi]_{C^{0,\alpha_c}(B_1)}
 =[v]_{C^{0,\alpha_c}(B_r(x_0))},~~
 \|D\psi\|_{L^{\gamma q_c}(B_1)}
 =\|Dv\|_{L^{\gamma q_c}(B_r(x_0))},
~~
 \|D^2\psi\|_{L^{q_c}(B_1)}
 =\|D^2v\|_{L^{q_c}(B_r(x_0))}.
\]
Substitution the identities shown above in \eqref{eq:GN-base} and raising the resulting inequality to
the power $\gamma$ prove \eqref{eq:critical-GN}, with a constant $C$ independent
of $r$.
\end{proof}

 For use in the blow-up argument proving Theorem~\ref{thm:main}, we now
establish the following a priori estimate for the linearized
Hamilton--Jacobi equation with a small drift.



\begin{lemma} \label{lem:small-drift}
Let $1<p<d$.  There are constants $\kappa_p>0$ and $C_p>0$ with the following
properties.

\begin{enumerate}
\item[(i).] If $\mathbf{b}\in L^d(B_2;\mathbb R^d)$,
$\|\mathbf{b}\|_{L^d(B_2)}\le\kappa_p$, and
$z\in W^{2,p}(B_2)$ solves
\begin{equation}\label{eq:drift-equation}
 -\Delta z+\mathbf{b}\cdot Dz=h\quad\text{in }B_2,
\end{equation}
then
\begin{equation}\label{eq:drift-estimate}
 \|D^2z\|_{L^p(B_1)}
 \le C_p\bigl(\|h\|_{L^p(B_2)}+\|z\|_{L^p(B_2)}\bigr).
\end{equation}

\item[(ii).] Suppose that $1<p_0<p<d$, $z\in W^{2,p_0}(B_2)\cap W^{1,p}(B_2)$ solves
\eqref{eq:drift-equation} with $h\in L^p(B_2)$, and  $\|\mathbf{b}\|_{L^d(B_2)}
 \le \min\{\kappa_{p_0},\kappa_p\}.
$
Then $ z\in W^{2,p}(B_1)$
and
\[
 \|z\|_{W^{2,p}(B_1)}
 \le C\bigl(
 \|h\|_{L^p(B_2)}
 +\|z\|_{W^{1,p}(B_2)}
 \bigr),
\]
where $C$ depends only on $d$, $p_0$, and $p$.
\end{enumerate}
\end{lemma}

\begin{proof}
For conclusion (i), we choose $1<r<s<2$ and a cutoff $\eta$ equal to one on $B_r$ and supported
in $B_s$.  Define $Z=\eta z$ and extend it by zero, then 
Calder\'on--Zygmund and Sobolev inequalities in the whole space give
\begin{align}\label{CZ_estimate}
 \|D^2Z\|_{L^p}
 \le C_p\|\Delta Z\|_{L^p},
 \qquad
 \|DZ\|_{L^{p^*}}\le C_p\|D^2Z\|_{L^p},
\end{align}
whre $C_p>0$ is a constant.
In light of (\ref{eq:drift-equation}), we obtain $Z$ satisfies
\[
 -\Delta Z+\mathbf{b}\cdot DZ
 =\eta h-2D\eta\cdot Dz-z\Delta\eta+\mathbf{b}\cdot zD\eta.
\]
Then by using (\ref{CZ_estimate}), we estimate the drift to obtain
\[
 \|\mathbf{b}\cdot DZ\|_{L^p}
 \le\|\mathbf{b}\|_{L^d}\|DZ\|_{L^{p^*}}
 \le C_p\|\mathbf{b}\|_{L^d}\|D^2Z\|_{L^p}.
\]
Moreover, the interpolation inequality implies
\[
 \|Dz\|_{L^p(B_s)}
 \le \theta(s-r)\|D^2z\|_{L^p(B_s)}
 +C_{p,\theta}(s-r)^{-1}\|z\|_{L^p(B_s)}.
\]
The term $z \mathbf{b}\cdot D\eta$ is treated similarly by using 
$\|z\|_{L^{p^*}}\le C(\|Dz\|_{L^p}+\|z\|_{L^p}).$  
Taking first $\kappa_p$ and then $\theta$ small gives
\[
 \|D^2z\|_{L^p(B_r)}\le\tfrac12\|D^2z\|_{L^p(B_s)}
 +C_p(s-r)^{-2}
   \bigl(\|h\|_{L^p(B_2)}+\|z\|_{L^p(B_2)}\bigr).
\]
We next use the hole-filling lemma to prove \eqref{eq:drift-estimate}.  Set
\[
 H:=\|h\|_{L^p(B_2)}+\|z\|_{L^p(B_2)}.
\]
Fix $q\in(2^{-1/2},1)$, for instance $q=3/4$, and define
\[
 r_k:=2-q^k,\qquad k=0,1,2,\ldots.
\]
Then $r_0=1$, $r_k\uparrow2$, and
\[
 r_{k+1}-r_k=(1-q)q^k.
\]
Applying the preceding estimate with $r=r_k$ and $s=r_{k+1}$ gives
\[ \|D^2z\|_{L^p(B_{r_k})}
 \le \frac12 \|D^2z\|_{L^p(B_{r_{k+1}})}  
 +C_p(1-q)^{-2}q^{-2k}H,
\]
where $C_p>0$ is a constant.
Iterating from $k=0$ to $N-1$, we obtain
\begin{align}\label{filling-hole-1}
\|D^2z\|_{L^p(B_1)}
 \le 2^{-N}\|D^2z\|_{L^p(B_{r_N})}
 +C_p(1-q)^{-2}
 \sum_{k=0}^{N-1}\left(\frac{1}{2q^2}\right)^k H.
\end{align}
Since $q>2^{-1/2}$, the right hand side above converges. Moreover,
$ \|D^2z\|_{L^p(B_{r_N})}\le \|D^2z\|_{L^p(B_2)}<\infty$ since $z\in W^{2,p}(B_2)$; hence
\[
 2^{-N} \|D^2z\|_{L^p(B_{r_N})}\rightarrow0
 \qquad\text{as }N\to\infty.
\]
Letting $N\to\infty$ in (\ref{filling-hole-1}), we conclude that
 \eqref{eq:drift-estimate} holds.

For conclusion (ii), similarly, we choose radii $1<r<\rho<\frac32$
and a cutoff function $\eta\in C_c^\infty(B_\rho)$ such that
$0\le\eta\le1$ and $\eta\equiv1$ on $B_r$.  For $1<s<d$, define  $X_s(B_\rho):=W^{2,s}(B_\rho)\cap W^{1,s}_0(B_\rho)$ and set $Z:=\eta z$,
then
\begin{align}
 (-\Delta+\mathbf {b}\cdot D)Z
 &=
 \eta h-2D\eta\cdot Dz-z\Delta\eta+(\mathbf{b}\cdot D\eta)z=:F \text{ in } B_\rho. \label{eq:localized-drift-equation}
\end{align}
Since $z\in W^{1,p}(B_2)$, the first three terms in $F$ belong to
$L^p(B_\rho)$.  Moreover, by H\"older's inequality and Sobolev embedding,
\[
 \|(\mathbf{b}\cdot D\eta)z\|_{L^p(B_\rho)}
 \le C\|\mathbf{b}\|_{L^d(B_\rho)}\|z\|_{L^{p^*}(B_\rho)}
 \le C\|\mathbf{b}\|_{L^d(B_\rho)}\|z\|_{W^{1,p}(B_\rho)},
\]
where $C>0$ is a constant.  It then follows that 
$F\in L^p(B_\rho)$.


We next prove that the linear drift operator  $L_b:=-\Delta+\mathbf{b}\cdot D$
is an isomorphism from $W^{2,p}(B_\rho)\cap W^{1,p}_0(B_\rho)$ onto $L^p(B_\rho)$ 
 when $\|\mathbf{b}\|_{L^d(B_\rho)}$ is sufficiently
small.  Firstly, the Calder\'on--Zygmund estimate  yields
\[
 \|w\|_{W^{2,p}(B_\rho)}
 \le C_0\|\Delta w\|_{L^p(B_\rho)}
 \qquad
 \text{for every }w\in W^{2,p}(B_\rho)\cap W^{1,p}_0(B_\rho).
\]
Since $
 -\Delta w=L_bw-\mathbf{b}\cdot Dw,$
H\"older's inequality and the Sobolev embedding
$W^{2,p}(B_\rho)\cap W^{1,p}_0(B_\rho)
\hookrightarrow W^{1,p^*}(B_\rho)$ give
\begin{align}\label{findfrom}
 \|w\|_{W^{2,p}(B_\rho)}
 &\le C_0\|L_bw\|_{L^p(B_\rho)}
      +C_0\|\mathbf{b}\cdot Dw\|_{L^p(B_\rho)}\nonumber\\
 &\le C_0\|L_bw\|_{L^p(B_\rho)}
      +C_0\|\mathbf{b}\|_{L^d(B_\rho)}
       \|Dw\|_{L^{p^*}(B_\rho)}\nonumber\\
 &\le C_0\|L_bw\|_{L^p(B_\rho)}
      +C_1\|\mathbf{b}\|_{L^d(B_\rho)}
       \|w\|_{W^{2,p}(B_\rho)},
\end{align}
where $C_0$ and $C_1$ are positive constants.
Choose $\kappa_p>0$ so that $C_1\kappa_p\le\frac12$.

If $\|\mathbf{b}\|_{L^d(B_\rho)}\le\kappa_p$, we find from (\ref{findfrom}) that
\begin{equation}\label{eq:dirichlet-drift-estimate-p}
 \|w\|_{W^{2,p}(B_\rho)}
 \le 2C_0\|L_bw\|_{L^p(B_\rho)}
 \qquad\text{for every }w \in W^{2,p}(B_\rho)\cap W^{1,p}_0(B_\rho).
\end{equation}
This estimate gives injectivity.  Moreover, to obtain surjectivity, let
$G_p:L^p(B_\rho)\to W^{2,p}(B_\rho)\cap W^{1,p}_0(B_\rho)$ be the inverse of the Dirichlet Laplacian:
\[
 -\Delta G_p(g)=g,\qquad G_p(g)=0  \quad\text{on }\partial B_\rho.
\]
For $F\in L^p(B_\rho)$, consider
\[
 \Phi(w):=G_p(F-b\cdot Dw).
\]
The same estimate as above shows that
\[
 \|\Phi(w_1)-\Phi(w_2)\|_{W^{2,p}(B_\rho)}
 \le C_1\|b\|_{L^d(B_\rho)}
 \|w_1-w_2\|_{W^{2,p}(B_\rho)}.
\]
Thus, $\Phi$ is a contraction when
$\|b\|_{L^d(B_\rho)}\le\kappa_p$.  Its fixed point
$\widetilde Z\in X_p(B_\rho)$ satisfies
\[
 L_b\widetilde Z=F
 \quad\text{in }B_\rho,
 \qquad
 \widetilde Z=0
 \quad\text{on }\partial B_\rho.
\]
The a-priori estimate implies
\begin{equation}\label{eq:localized-solution-bound}
 \|\widetilde Z\|_{W^{2,p}(B_\rho)}
 \le C\|F\|_{L^p(B_\rho)}
 \le C\bigl(\|h\|_{L^p(B_2)}
             +\|z\|_{W^{1,p}(B_2)}\bigr),
\end{equation}
where $C>0$ is a constant.

We now compare $\widetilde Z$ with the original solution $Z$ to (\ref{eq:localized-drift-equation}).
Since $p>p_0$,
\[
 \widetilde Z\in W^{2,p}(B_\rho)\subset W^{2,p_0}(B_\rho).
\]
Therefore, $ Y:=Z-\widetilde Z
 \in X_{p_0}(B_\rho)$ and $ L_bY=0   ~\text{in }B_\rho.$
Applying  estimate \eqref{eq:dirichlet-drift-estimate-p} with
the exponent $p_0$  and using
$\|\mathbf{b}\|_{L^d(B_\rho)}\le\kappa_{p_0}$ gives
\[
 \|Y\|_{W^{2,p_0}(B_\rho)}
 \le C\|L_bY\|_{L^{p_0}(B_\rho)}=0.
\]
  Hence $Y=0$, so $Z=\widetilde Z\in W^{2,p}(B_\rho)$.  Since
$\eta\equiv1$ on $B_r$, \eqref{eq:localized-solution-bound} yields
\[
 \|z\|_{W^{2,p}(B_r)}
 \le C\bigl(\|h\|_{L^p(B_2)}
             +\|z\|_{W^{1,p}(B_2)}\bigr).
\]
Taking $r>1$ proves the desired interior regularity and estimate on $B_1$.
This completes the proof of this lemma.

\end{proof}
 
To derive the contradiction in the blow-up argument, we prove the following
Liouville theorem for \eqref{eq:intro-HJ}.
\begin{lemma} 
\label{lem:endpoint-liouville}
Let $a\ge0$ and
$v\in W^{2,q_c}_{\mathrm{loc}}(\R^d)\cap C^{0,\alpha_c}(\R^d)$ 
solve
\begin{equation}\label{eq:entire-equation}
 -\Delta 
 v+a|Dv|^\gamma=0\quad\text{a.e. in }\R^d.
\end{equation}
Then, $v$ is constant.
\end{lemma}

\begin{proof}
  If $a=0$, then $v$ is
harmonic and, for every $x_0\in\mathbb R^d$ and $R>0$, the interior gradient
estimate gives
\[
 |Dv(x_0)|
 \le \frac{C}{R}\sup_{B_R(x_0)}|v-v(x_0)|
 \le C [v]_{C^{0,\alpha_c}(\mathbb R^d)}R^{\alpha_c-1}.
\]
Letting $R\to\infty$ and using $\alpha_c<1$ yields $Dv(x_0)=0$.

Assume now that $a>0$ and define $\mathbf b:=a|Dv|^{\gamma-2}Dv$.  The Sobolev embedding
$W^{2,q_c}_{\mathrm{loc}}\hookrightarrow
W^{1,\gamma q_c}_{\mathrm{loc}}$ and the identity
$\gamma q_c=d(\gamma-1)$ imply
$\mathbf b\in L^d_{\mathrm{loc}}(\mathbb R^d)$.  Fix $x_0\in\mathbb R^d$, by absolute continuity, one can choose $r=r(x_0)>0$ such that $\|\mathbf b\|_{L^d(B_{2r}(x_0))}
 \le \min\{\kappa_{q_c},\kappa_p\},$
where $p$ is chosen so that
\begin{equation}\label{eq:liouville-p-choice}
 \max\left\{q_c,\frac{d\gamma}{\gamma+1}\right\}<p<d,
\end{equation}
where we have used
$q_c<d$ and $d\gamma/(\gamma+1)<d$.  Moreover,
$\gamma q_c=d(\gamma-1)>d>p$, hence
$v\in W^{1,p}_{\mathrm{loc}}$.

After translating $x_0$ and scaling $B_{2r}(x_0)$ to $B_2$, equation
\eqref{eq:entire-equation} becomes
$-\Delta v+\mathbf b\cdot Dv=0$, and the $L^d$ norm of the  drift is scaling invariant.  Invoking  Lemma~\ref{lem:small-drift}(ii), with $p_0=q_c$,   we obtain $v\in W^{2,p}_{\mathrm{loc}}(\mathbb R^d).$
Consequently, $Dv\in L^{p^*}_{\mathrm{loc}}(\mathbb R^d)$,
 $p^*:=\frac{dp}{d-p}$,
and \eqref{eq:liouville-p-choice} is exactly the condition $s:=\frac{p^*}{\gamma}>d.$
It follows from \eqref{eq:entire-equation} that
$\Delta v=a|Dv|^\gamma\in L^s_{\mathrm{loc}}$.  Interior
Calder\'on--Zygmund estimates yield $ v\in W^{2,s}_{\mathrm{loc}}(\mathbb R^d)$\text{ for some }$s>d$.
Thus, all the regularity hypotheses of the superquadratic Liouville theorem
\cite[Lemma~2.5]{CirantVerzini2022} are satisfied, which  implies that $v$ is constant.
\end{proof}

\section{Uniformly H\"{o}lder vanishing estimate} \label{Section3}

Lemma~\ref{lem:individual-little} yields the vanishing of the 
H\"older quotient for each fixed function
$u\in W^{2,q_c}(\mathbb T^d)$.  In this section, we strengthen
\eqref{eq:individual-little} to a uniform estimate over the family of
solutions under consideration.

Let $\iota_0>0$ denote the injectivity radius of the flat torus.  For
$u\in W^{2,q_c}(\mathbb T^d)$ and $0<r<\iota_0$, define
\begin{equation}\label{eq:exact-holder-profile}
 H_u(r):=
 \max_{\operatorname{dist}(x,y)=r}
 \frac{|u(x)-u(y)|}{r^{\alpha_c}},
\end{equation}
and set
\begin{equation}\label{eq:theta}
 \Theta_u(r):=\max_{0<t\le r}H_u(t).
\end{equation}
The maxima shown above are attained.  Indeed, for $r<\iota_0$, every pair at
distance $r$ can be written uniquely as
$y=\exp_x(r\omega)$ with $x\in\mathbb T^d$ and
$\omega\in\mathbb S^{d-1}$.  Hence
\[
 H_u(r)=
 \max_{(x,\omega)\in\mathbb T^d\times\mathbb S^{d-1}}
 \frac{|u(\exp_x(r\omega))-u(x)|}{r^{\alpha_c}}.
\]
The maximized function above is continuous jointly in
$(r,x,\omega)$ and the parameter space
$\mathbb T^d\times\mathbb S^{d-1}$ is compact.  Therefore, $H_u$ is
continuous on $(0,\iota_0)$.  Lemma~\ref{lem:individual-little} shows that
$H_u(r)\to0$ as $r\downarrow0$; after setting $H_u(0)=0$, it is continuous
on $[0,\iota_0)$.  $\Theta_u$ is consequently
continuous and non-decreasing.  Hence, the two maxima can be attained.

Now, assume, by contradiction, that the desired conclusion for uniformly vanishing property fails.  Then
there exist $\delta>0$, $u_n\in\mathcal U$, and $s_n\downarrow0$ such that 
 $s_n<\iota_0$ and $\Theta_{u_n}(s_n)\ge\delta.$
Fix  $0<\varepsilon<\delta/2$.  Since
$\Theta_{u_n}(0)=0$, define
\begin{equation}\label{eq:first-crossing-definition}
 r_n:=\min\bigl\{r\in[0,s_n]:
                 \Theta_{u_n}(r)=\varepsilon\bigr\}.
\end{equation}
Continuity and monotonicity give $0<r_n\le s_n,$ $\Theta_{u_n}(r_n)=\varepsilon$ and $\Theta_{u_n}(r)<\varepsilon\quad(0\le r<r_n).$ Noting that the maximum defining $\Theta_{u_n}(r_n)$ cannot occur at a radius
$t<r_n$,  otherwise it  would give $\Theta_{u_n}(t)=\varepsilon$ and
contradict the minimality in \eqref{eq:first-crossing-definition}.  Hence
$H_{u_n}(r_n)=\varepsilon$.  By \eqref{eq:exact-holder-profile}, there are
$x_n,y_n\in\mathbb T^d$ such that $\operatorname{dist}(x_n,y_n)=r_n$,
 $|u_n(x_n)-u_n(y_n)|=\varepsilon r_n^{\alpha_c}.$
In particular, $r_n\downarrow0$.  The contradiction argument below will require the coefficient $ \lambda=\varepsilon^{\gamma-1}$
in the rescaled equation to be sufficiently small.  We postpone the final
choice of $\varepsilon$ until the universal threshold $\lambda_*>0$ is
obtained in Section~\ref{sec:first-compactness}.  Notice, however, that the
following first crossing construction is valid for every
\begin{align}\label{eps-small}
 0<\varepsilon<\frac{\delta}{2}.
\end{align}
Consequently, once $\lambda_*$ is known, we may choose
\[
 0<\varepsilon<
 \min\left\{\frac{\delta}{2},\lambda_*^{1/(\gamma-1)}\right\}
\]
and apply the construction below with this choice of $\varepsilon$.

For all large $n$, let $\zeta_n\in T_{x_n}\mathbb T^d\simeq\mathbb R^d$
be the unique vector with
$y_n=\exp_{x_n}(\zeta_n)$ and $|\zeta_n|=r_n$, and define
\[
 z_n:=\frac{\zeta_n}{r_n},
 \qquad |z_n|=1.
\]
Using the flat exponential coordinates based at $x_n$, introduce the
rescaling:
\begin{equation}\label{eq:first-rescaling}
 v_n(y)=\frac{u_n(x_n+r_ny)-u_n(x_n)}{\varepsilon r_n^{\alpha_c}},
 \qquad
 g_n(y)=\varepsilon^{-1}r_n^{\gamma'}f_n(x_n+r_ny).
\end{equation}
With (\ref{eq:first-rescaling}), (\ref{eq:intro-HJ}) gives us
\begin{equation*}
 -\Delta v_n+\lambda|Dv_n|^\gamma=g_n\text{ in }\mathbb T_n^d,
 \qquad \lambda=\varepsilon^{\gamma-1},
\end{equation*}
where $
 \mathbb T_n^d:=r_n^{-1}(\mathbb T^d-x_n).$
Moreover,
\begin{equation}\label{eq:first-normalization}
 v_n(0)=0,\qquad |v_n(z_n)|=1,
\end{equation}
and, whenever $|y-z|\le1$,
\begin{equation}\label{eq:unit-holder}
 |v_n(y)-v_n(z)|\le|y-z|^{\alpha_c}.
\end{equation}

We next show the vanishing of $g_n$ after translation and scaling, which is
\begin{lemma}\label{lem:data-vanish}
For every fixed $R>0$,
\begin{equation*}
 g_n\rightarrow0
 \quad\text{ in }L^{q_c}(B_R).
\end{equation*}
Moreover, for every fixed $L>0$,
\begin{equation}\label{eq:secondary-data-vanish}
 \sup_{\substack{\xi\in\mathbb T_n^d\\0<\rho\le1}}
 \bigl\|\rho^{\gamma'}g_n(\xi+\rho\,\cdot)\bigr\|_{L^{q_c}(B_L)}
\rightarrow0.
\end{equation}
\end{lemma}

\begin{proof}
We have from \eqref{eq:first-rescaling} and (\ref{eq:UInt}) that
\[
 \|g_n\|_{L^{q_c}(B_R)}^{q_c}
 =\varepsilon^{-q_c}
   \int_{B_{Rr_n}(x_n)}|f_n|^{q_c}\,d x
 \le\varepsilon^{-q_c}\omega_{\mathcal F}(|B_R|r_n^d)\longrightarrow0.
\]
Moreover, a direct computation implies
\begin{align*}
 &\bigl\|\rho^{\gamma'}
        g_n(\xi+\rho\,\cdot)\bigr\|_{L^{q_c}(B_L)}^{q_c}\\
 &\qquad={}
 \varepsilon^{-q_c}
 \int_{B_{Lr_n\rho}(x_n+r_n\xi)}|f_n|^{q_c}\,d x
 \le
 \varepsilon^{-q_c}\omega_{\mathcal F}(|B_L|(r_n\rho)^d)
 \le
 \varepsilon^{-q_c}\omega_{\mathcal F}(|B_L|r_n^d)
 \longrightarrow0.
\end{align*}
\end{proof}

We note that Theorem~\ref{thm:critical-holder} also yields
\begin{equation}\label{eq:first-global-holder}
 [v_n]_{C^{0,\alpha_c}(\mathbb T_n^d)}
 \le \frac{K}{\varepsilon}.
\end{equation}
In particular, since $v_n(0)=0$, for every fixed $R>0$, we have
\[
 \|v_n\|_{L^\infty(B_R)}
 \le \frac{K}{\varepsilon}R^{\alpha_c}.
\]

Our next goal is to establish the local compactness of $v_n$.  To this
end, in the following subsection we prove the uniform equi-integrability of
the local energy.

\subsection{ Uniformly local energy bound}

For a function $v\in W^{2,q_c}(G)$ and a measurable set $E\subset G$, set
\begin{equation}\label{eq:energy}
 \mathcal E(v;E)=
 \int_E\left(|D^2v|^{q_c}+|Dv|^{\gamma q_c}\right)\,d x.
\end{equation}
Then we prove the following energy estimate.
\begin{proposition}
\label{prop:energy-bound}
Define  $\mathbb T_n^d=r_n^{-1}(\mathbb T^d-x_n)$ and  suppose
$v_n\in W^{2,q_c}(\mathbb T_n^d)$ satisfies
\begin{equation}\label{eq:energy-prop-equation}
 -\Delta v_n+\lambda_n|Dv_n|^\gamma=g_n,
 \qquad 0\le\lambda_n\le1,
\end{equation}
where $g_n\to0$  in $L^{q_c}_{\mathrm{loc}}(\R^d)$, and assume
\eqref{eq:unit-holder} holds on $\mathbb T_n^d$.  Suppose that 
\begin{equation}\label{eq:abstract-secondary-vanish}
 \sup_{\substack{\xi\in\mathbb T_n^d\\0<\rho\le1}}
 \bigl\|\rho^{\gamma'}g_n(\xi+\rho\,\cdot)\bigr\|_{L^{q_c}(B_L)}
 \longrightarrow0
\end{equation}
for every fixed $L$.  Then,  \begin{equation}\label{eq:energy-bound}
 \sup_n\mathcal E(v_n;B_R)<\infty.
\end{equation}

Moreover, there is a constant
$M_0=M_0(d,\gamma)$ such that, for every fixed $R$, all sufficiently large
$n$ satisfy
\begin{equation}\label{eq:universal-tail-energy}
 \sup_{x\in B_R}\mathcal E(v_n;B_2(x))\le M_0.
\end{equation}
\end{proposition}

\begin{proof}
Assume that the conclusion (\ref{eq:energy-bound}) fails.  After passing to a subsequence, there is
a fixed ball $B_R$ on which the energies $\mathcal E(v_n;B_R)\rightarrow\infty$.  Fix a small
number $\eta>0$, to be chosen below.  Define, for radii $0<\rho\le1$,
\[
 Q_n(\rho)=
 \sup_{x\in\mathbb T_n^d}\mathcal E(v_n;B_\rho(x)).
\]
For each fixed $n$, uniform absolute continuity on $\mathbb T_n^d$ gives
$Q_n(\rho)\to0$ as $\rho\downarrow0$.  The function $Q_n$ is continuous:
translations and scalings are continuous in $L^1$ for the integrable energy density, uniformly over the compact set of centres.  A finite
covering of $B_R$ shows that $Q_n(1)\to\infty$.  Consequently, there are
$\rho_n\in(0,1)$ and $\xi_n\in\mathbb T_n^d$ such that
\begin{equation}\label{eq:energy-normalization}
 Q_n(\rho_n)=\eta,\qquad
 \mathcal E(v_n;B_{\rho_n}(\xi_n))\ge\eta/2.
\end{equation}
The radii satisfy $\rho_n\to0$: if not, then $\rho_n\ge\rho_*>0$ along a
subsequence, then $Q_n(\rho_*)\le\eta$, and a fixed finite covering of
$B_R$ by $\rho_*$-balls would give a uniform bound for its energy, which is a contradiction with 
 the choice of $B_R$.

Define the   blow-up sequence
\begin{equation}\label{eq:second-rescaling}
 w_n(y)=
 \frac{v_n(\xi_n+\rho_ny)-v_n(\xi_n)}{\rho_n^{\alpha_c}},
 \qquad
 k_n(y)=\rho_n^{\gamma'}g_n(\xi_n+\rho_ny).
\end{equation}
By using (\ref{eq:energy-prop-equation}), we find it solves
\begin{equation}\label{eq:second-equation}
 -\Delta w_n+\lambda_n|Dw_n|^\gamma=k_n.
\end{equation}
For every fixed $L$, a finite covering by unit balls and
\eqref{eq:energy-normalization} give
\begin{equation}\label{eq:second-energy-bound}
 \sup_n\mathcal E(w_n;B_L)\le C_L\eta,
 \qquad
 \mathcal E(w_n;B_1)\ge\eta/2.
\end{equation}
Furthermore, $k_n\to0$  in $L^{q_c}_{\mathrm{loc}}$ by
\eqref{eq:abstract-secondary-vanish}.  
  Since $\rho_n\to0$, \eqref{eq:unit-holder} becomes, for each
fixed $L$ and all large $n$,
\begin{equation}\label{eq:second-global-holder}
 [w_n]_{C^{0,\alpha_c}(B_L)}\le1.
\end{equation}
Hence, a diagonal subsequence converges locally uniformly to a 
$\alpha_c$-H\"older function $w$, and we may assume
$\lambda_n\to\lambda_\infty\in[0,1]$.

It remains to prove strong convergence of $w_n$ when $q=q_c$.  Let 
$z_{nm}=w_n-w_m$ and define
\begin{align}\label{Anm_label}
 A_{nm}(D w_n,Dw_m)=\gamma\int_0^1
 |Dw_m+t(Dw_n-Dw_m)|^{\gamma-2}
 \bigl(Dw_m+t(Dw_n-Dw_m)\bigr)\,d t.
\end{align}
Then, we have from (\ref{eq:second-equation}) that 
\begin{equation}\label{eq:second-difference}
 -\Delta z_{nm}+\mathbf{b}_{nm}\cdot Dz_{nm}=h_{nm},
 \quad \mathbf{b}_{nm}=\lambda_nA_{nm},
 \quad h_{nm}=k_n-k_m+(\lambda_m-\lambda_n)|Dw_m|^\gamma.
\end{equation}
Using $\gamma q_c=d(\gamma-1)$ and
\eqref{eq:second-energy-bound}, we find
\begin{equation}\label{eq:small-second-drift}
 \|\mathbf{b}_{nm}\|_{L^d(B_2)}
 \le C\left(
 \|Dw_n\|_{L^{\gamma q_c}(B_2)}^{\gamma-1}
 +\|Dw_m\|_{L^{\gamma q_c}(B_2)}^{\gamma-1}
 \right)
 \le C\eta^{(\gamma-1)/(\gamma q_c)}.
\end{equation}
Choose $\eta$ so small that the last quantity is less than
$\kappa_{q_c}$ in Lemma \ref{lem:small-drift}.   For every fixed $L$, the first term in $h_{nm}$ satisfies $ \|k_n-k_m\|_{L^{q_c}(B_L)}\rightarrow0.$  Moreover, by \eqref{eq:second-energy-bound}, $ \bigl\||Dw_m|^\gamma\bigr\|_{L^{q_c}(B_L)}
 \le C_L\eta^{1/q_c},$
whereas $\lambda_m-\lambda_n\to0$.  Hence, $\|h_{nm}\|_{L^{q_c}(B_L)}\rightarrow0.$
Since $z_{nm}\to0$ locally uniformly, Lemma~\ref{lem:small-drift} 
indicates $w_n\rightarrow w$
in $W^{2,q_c}_{\mathrm{loc}}(\R^d).$ Indeed, cover \(B_L\) by finitely many balls
\(B_1(x_j)\) with \(B_2(x_j)\subset B_{L+2}\).  The normalization
\(Q_n(\rho_n)=\eta\) and a finite covering of each \(B_2(x_j)\) by unit
balls conclude $\mathcal E(w_n;B_2(x_j))\le C_d\eta .$
Hence, choosing \(\eta\) sufficiently small, we have  $\|\mathbf b_{nm}\|_{L^d(B_2(x_j))}\le\kappa_{q_c}.$
Conclusion (i) in Lemma~\ref{lem:small-drift} therefore yields
\[
 \|D^2z_{nm}\|_{L^{q_c}(B_1(x_j))}
 \le C\bigl(
 \|h_{nm}\|_{L^{q_c}(B_2(x_j))}
 +\|z_{nm}\|_{L^{q_c}(B_2(x_j))}\bigr)\to0.
\]
Summing over \(j\), and using \(z_{nm}\to0\) locally uniformly together
with interpolation inequalities, shows that \((w_n)_n\) is a Cauchy sequence in
\(W^{2,q_c}(B_L)\).  Thus,  $w_n\to w$ strongly in  $W^{2,q_c}_{\mathrm{loc}}(\mathbb R^d)$. Moreover, Sobolev inequality with$q_c^*=\gamma q_c$ implies
$Dw_n\to Dw$ in $L^{\gamma q_c}_{\mathrm{loc}}$.

Now, we pass to the limit in \eqref{eq:second-equation} and obtain
\[
 -\Delta w+\lambda_\infty|Dw|^\gamma=0\quad\text{in }\R^d.
\]
By \eqref{eq:second-global-holder}, $w$ has a finite 
$C^{0,\alpha_c}$ semi-norm. With the aid of  Lemma \ref{lem:endpoint-liouville}, we have $w$ is a  constant.
On the other hand, strong convergence of $w_n$ and
\eqref{eq:second-energy-bound} imply
$\mathcal E(w;B_1)\ge\eta/2$, which is a contradiction.  Thus
\eqref{eq:energy-bound} holds.

It remains to prove that the constant in
\eqref{eq:universal-tail-energy} can be chosen uniformly.  Suppose this is false.  Then, for every integer
$j\ge1$, there exist a sequence $\bigl(v_n^{(j)},g_n^{(j)},\lambda_n^{(j)},
       \mathbb T_n^{(j),d}\bigr)_{n\in\mathbb N}$
satisfying all the hypotheses of the proposition, a radius $R_j>0$, and
infinitely many indices $n$ such that
\begin{equation}\label{eq:universal-failure}
 \sup_{\xi\in B_{R_j}}
 \mathcal E\bigl(v_n^{(j)};B_2(\xi)\bigr)>j.
\end{equation}
For each $j$, choose one such index $n_j\ge j$ sufficiently large that the
injectivity radius of $\mathbb T_{n_j}^{(j),d}$ is larger than $j+4$ and
\begin{equation}\label{eq:diagonal-data-small}
 \sup_{\substack{\xi\in\mathbb T_{n_j}^{(j),d}\\0<\rho\le1}}
 \left\|\rho^{\gamma'}g_{n_j}^{(j)}
            (\xi+\rho\,\cdot)\right\|_{L^{q_c}(B_L)}
 \le\frac1j
 \quad\text{for every integer }1\le L\le j.
\end{equation}
By \eqref{eq:universal-failure}, choose $\xi_j\in B_{R_j}$ such that
\begin{equation}\label{eq:diagonal-high-energy}
 \mathcal E\bigl(v_{n_j}^{(j)};B_2(\xi_j)\bigr)>j.
\end{equation}
 Define
\[
 \widetilde{\mathbb T}_j^d
 :=\mathbb T_{n_j}^{(j),d}-\xi_j,
 \qquad
 \widetilde v_j(y)
 :=v_{n_j}^{(j)}(\xi_j+y)-v_{n_j}^{(j)}(\xi_j),
\]
\[
 \widetilde g_j(y):=g_{n_j}^{(j)}(\xi_j+y),
 \qquad
 \widetilde\lambda_j:=\lambda_{n_j}^{(j)},
\]
then
\[
 -\Delta\widetilde v_j
 +\widetilde\lambda_j|D\widetilde v_j|^\gamma
 =\widetilde g_j
 \quad\text{in }\widetilde{\mathbb T}_j^d,
 \qquad 0\le\widetilde\lambda_j\le1.
\]
We have  $\bigl(\widetilde v_j,\widetilde g_j,
       \widetilde\lambda_j,\widetilde{\mathbb T}_j^d\bigr)$ satisfies all conditions where the first part of the proof
applies.  Taking $R=2$ in \eqref{eq:energy-bound} gives $ \sup_j\mathcal E(\widetilde v_j;B_2)<\infty.$
This contradicts \eqref{eq:diagonal-high-energy}, which leads to
$\mathcal E(\widetilde v_j;B_2)>j$.  Therefore, the  constant
$M_0=M_0(d,\gamma)$ exists, and \eqref{eq:universal-tail-energy} follows.
\end{proof}

\subsection{Strong compactness and the equi-integrability of gradient estimates}
\label{sec:first-compactness}
The energy bound yields only weak compactness for the first blow-up
sequence.  The small coefficient $\lambda=\varepsilon^{\gamma-1}$
arising from the first rescaling allows us to upgrade this weak compactness
to strong local compactness.
\begin{proposition} 
\label{prop:first-strong}
Suppose that there exists a constant $\lambda_*=\lambda_*(d,\gamma)>0$  such that  $v_n$ satisfies the hypotheses of Proposition
\ref{prop:energy-bound}  with 
$0<\lambda\le\lambda_*$.  Then, after passing to a subsequence,
\begin{equation}\label{eq:first-strong}
 v_n\rightarrow v
 \quad\text{in }W^{2,q_c}_{\mathrm{loc}}(\R^d),
 \qquad
 Dv_n\rightarrow Dv
 \quad\text{ in }L^{\gamma q_c}_{\mathrm{loc}}(\R^d),
\end{equation}
where $v_n(0)=0.$
\end{proposition}

\begin{proof}
\(\eqref{eq:unit-holder}\) together with $v_n(0)=0$ yields uniform bounds
for $v_n$ on every fixed ball.  Hence, up to a subsequence,
\[
 v_n\rightarrow v
 \qquad\text{locally uniformly in }\mathbb R^d.
\]
For $L>0$, we apply   Proposition~\ref{prop:energy-bound}   with $R=L+2$ and obtain that for all sufficiently large $n$,
\begin{equation}\label{eq:first-energy-on-compact}
 \sup_{x\in B_{L+2}}
 \mathcal E(v_n;B_2(x))\le M_0.
\end{equation}
Set $z_{nm}:=v_n-v_m$.  Subtracting the two equations yields
\begin{equation}\label{eq:first-linearized-difference}
 -\Delta z_{nm}+\mathbf{b}_{nm}\cdot Dz_{nm}=g_n-g_m,
 \qquad
 \mathbf{b}_{nm}=\lambda A_{nm}(Dv_n,Dv_m),
\end{equation}
where $A_{nm}$ is defined in \eqref{Anm_label}.  Since $|A_{nm}|
 \le C_\gamma\bigl(|Dv_n|^{\gamma-1}
                    +|Dv_m|^{\gamma-1}\bigr)$
and $d(\gamma-1)=\gamma q_c$, we obtain that for every
$x\in B_L$,
\begin{align}
 \|\mathbf b_{nm}\|_{L^d(B_2(x))}
 &\le C_\gamma\lambda
 \left(
 \|Dv_n\|_{L^{\gamma q_c}(B_2(x))}^{\gamma-1}
 +\|Dv_m\|_{L^{\gamma q_c}(B_2(x))}^{\gamma-1}
 \right)\notag\\
 &\le 2C_\gamma\lambda M_0^{1/d}.
 \label{eq:first-uniform-drift-smallness}
\end{align}
Choose $\lambda_*:=
 \min\left\{1,\frac{\kappa_{q_c}}
 {4C_\gamma M_0^{1/d}}\right\},$
then \eqref{eq:first-uniform-drift-smallness} is strictly below
$\kappa_{q_c}$ whenever $0<\lambda\le\lambda_*$.  Since
$g_n-g_m\to0$ in $L^{q_c}$ on every fixed ball and
$z_{nm}\to0$ locally uniformly, we invoke conclusion (i) in Lemma~\ref{lem:small-drift} and obtain
\[
 \|D^2z_{nm}\|_{L^{q_c}(B_1(x))}
 \le C\left(
 \|g_n-g_m\|_{L^{q_c}(B_2(x))}
 +\|z_{nm}\|_{L^{q_c}(B_2(x))}
 \right)\longrightarrow0.
\]
A finite covering of $B_L$ by the unit balls proves that $(v_n)_n$ is a 
Cauchy sequence in $W^{2,q_c}(B_L)$.  Since $L$ is arbitrary,
$v_n\to v$   in $W^{2,q_c}_{\mathrm{loc}}$.  Finally,  
Sobolev inequality and $q_c^*=\gamma q_c$ imply
$Dv_n\to Dv$   in $L^{\gamma q_c}_{\mathrm{loc}}$.
\end{proof}

We now select $\varepsilon$ in \eqref{eps-small} so small that
\begin{equation}\label{eq:epsilon-final}
 \varepsilon^{\gamma-1}\le\lambda_*.
\end{equation}
Recall that  $\varepsilon$ was used in the beginning of  Section \ref{Section3}, we simply repeat
the first-crossing construction with this final value.



Now, we are ready to establish the uniformly H\"older vanishing estimate of the solution $u$ to (\ref{eq:intro-HJ}).

\begin{theorem} \label{thm:uniform-little}
Under the conditions of Theorem \ref{thm:main}, we have
\begin{equation}\label{eq:uniform-little}
 \lim_{r\downarrow0}
 \sup_{u\in\mathcal U}\ \sup_{0<\text{dist}(x,y)\le r}
 \frac{|u(x)-u(y)|}{\text{dist}(x,y)^{\alpha_c}}=0.
\end{equation}
\end{theorem}

\begin{proof}
If \eqref{eq:uniform-little} does not hold, the first blow-up sequence in Section \ref{Section3}
would exist with $\varepsilon$ satisfying \eqref{eq:epsilon-final}.  Lemma
\ref{lem:data-vanish} and Proposition \ref{prop:first-strong} yield, after
extraction,
\[
 v_n\to v\quad\text{locally uniformly and strongly in }
 W^{2,q_c}_{\mathrm{loc}}(\mathbb R^d).
\]
Therefore,
\[
 -\Delta v+\varepsilon^{\gamma-1}|Dv|^\gamma=0\quad\text{in }\R^d.
\]
In particular, the global bound \eqref{eq:first-global-holder} passes to $v$.  Moreover, \eqref{eq:first-normalization} gives, along a subsequence
$z_n\to z$ with $|z|=1$,
\[
 v(0)=0,\qquad |v(z)|=1,
\]
so $v$ is non-constant. Whereas, Lemma \ref{lem:endpoint-liouville} implies that $v$ is
constant.  This contradiction proves \eqref{eq:uniform-little}.
\end{proof}

\section{Maximal regularity: blow-up estimates}\label{sec:maximal-regularity}

In this section, we first establish the gradient estimate by using Theorem \ref{thm:uniform-little}. %

\begin{lemma}
\label{lem:global-interpolation}
Let $a>0$ be arbitrary.  Suppose $u\in W^{2,q_c}(\mathbb T^d)$ and there is $r_0>0$ such that
\begin{equation}\label{eq:local-small-holder}
 \sup_{0<\text{dist}(x,y)\le4r_0}
 \frac{|u(x)-u(y)|}{\text{dist}(x,y)^{\alpha_c}}\le a.
\end{equation}
Then
\begin{equation}\label{eq:global-GN}
 \bigl\||Du|^\gamma\bigr\|_{L^{q_c}(\mathbb T^d)}
 \le C a^{\gamma-1}\|D^2u\|_{L^{q_c}(\mathbb T^d)}
 +C{r_0}^{-\gamma'}a^\gamma,
\end{equation}
where $C=C(d,\gamma)$ is a positive constant.
\end{lemma}

\begin{proof}
Choose $r_0>0$ smaller than one quarter of the radius of
$\mathbb T^d$.  There exist points $x_1,\dots,x_N\in\mathbb T^d$ such that $\mathbb T^d\subset\bigcup_{j=1}^N B_{r_0}(x_j),$
and the family of doubled balls $ \bigl\{B_{2r_0}(x_j)\bigr\}_{j=1}^N$
has overlap bounded by a constant depending only on $d$. The cover is  chosen such that
\begin{equation}\label{eq:cover-cardinality}
 N\le C  r_0^{-d}.
\end{equation}
For $x,y\in B_{2r_0}(x_j)$, one has
$\operatorname{dist}(x,y)\le4r_0$; hence,
\eqref{eq:local-small-holder} implies $[u]_{C^{0,\alpha_c}(B_{2r_0}(x_j))}\le a.$
We apply 
Lemma~\ref{lem:critical-GN} to $u$ on each ball $B_{2r_0}(x_j)$ and obtain
\[
 \|Du\|_{L^{\gamma q_c}(B_{r_0}(x_j))}^{\gamma}
 \le
 C a^{\gamma-1}
 \|D^2u\|_{L^{q_c}(B_{2r_0}(x_j))}
 +C a^{\gamma},
\]
where the same constant $C=C(d,\gamma)$ is used for every $j$ and every
$r_0$.
Raising this inequality to the power $q_c$ and summing in $j$   yield
\[
 \|Du\|_{L^{\gamma q_c}(\mathbb T^d)}^{\gamma q_c}
 \le C a^{(\gamma-1)q_c}
       \|D^2u\|_{L^{q_c}(\mathbb T^d)}^{q_c}
      +CN a^{\gamma q_c}.
\]
Then, we take the $q_c$-th root in the above inequality, use
$d/q_c=\gamma'$ and \eqref{eq:cover-cardinality} to get
\[
 \bigl\||Du|^\gamma\bigr\|_{L^{q_c}(\mathbb T^d)}
 \le C a^{\gamma-1}\|D^2u\|_{L^{q_c}(\mathbb T^d)}
      +C r_0^{-d/q_c}a^\gamma
 =C a^{\gamma-1}\|D^2u\|_{L^{q_c}}
      +C r_0^{-\gamma'}a^\gamma,
\]
which proves \eqref{eq:global-GN} and the coefficient of the
Hessian term is independent of $r_0$.
 \end{proof}

Now, we are ready to prove the main Theorem \ref{thm:main}.
\begin{proof}[Proof of Theorem \ref{thm:main}]
Let
$M=\sup_{f\in\mathcal F}\|f\|_{L^{q_c}}$ given in \eqref{eq:intro-HJ}.  The 
Calder\'on--Zygmund estimate in $\mathbb T^d$ and the zero average of the solution to (\ref{eq:intro-HJ}) give
\begin{equation}\label{eq:global-CZ-final}
 \|u\|_{W^{2,q_c}}
 \le C\|\Delta u\|_{L^{q_c}}
 \le C\left(M+\left\||Du|^\gamma\right\|_{L^{q_c}}\right),
\end{equation}
where $C>0$ is the Calder\'on--Zygmund constant. 
Choose $a>0$ sufficiently small such  that the product of 
$Ca^{\gamma-1}$ in (\ref{eq:global-GN}) of Lemma \ref{lem:global-interpolation} and the
Calder\'on--Zygmund constant in \eqref{eq:global-CZ-final} is less than $1/2$.
Invoking Theorem \ref{thm:uniform-little}, we obtain a  constant $r_0>0$ for which
\eqref{eq:local-small-holder} holds for every $u\in\mathcal U$ solving \eqref{eq:intro-HJ}.  Combining
\eqref{eq:global-GN} and \eqref{eq:global-CZ-final}, and absorbing the
Hessian term, gives a uniform bound for $\|u\|_{W^{2,q_c}}$.  The equation
then gives the same bound for
$\||Du|^\gamma\|_{L^{q_c}}$.  This finishes the  proof of 
\eqref{eq:main-estimate}.
\end{proof}

\section{Application to mean-field games systems at the critical coupling exponent}\label{section5}
In this section, we apply Theorem~\ref{thm:main} to establish the existence of solutions to stationary defocusing mean-field games systems with superquadratic Hamiltonians and critical coupling exponents.   

Recall that the Legendre transform of $H(p)=|p|^\gamma$ is $L(q)=c_\gamma|q|^{\gp}$,
 $c_\gamma=(\gamma-1)\gamma^{-\gp}.$  Define 
\begin{equation}\label{eq:perspective}
 \mathscr L(m,w)=
 \begin{cases}
 mL(-w/m)=c_\gamma |w|^{\gp}m^{1-\gp},&m>0,\\
 0,&(m,w)=(0,0),\\
 +\infty,&m=0,\ w\ne0,
 \end{cases}
\end{equation}
which satisfies the dual representation 
\begin{align}\label{dual_representation}
\mathscr L(m,w)=\sup_{p\in\R^d}\bigl(-w\cdot p-mH(p)\bigr).
\end{align}
Then we give the definition of the variational solution to (\ref{eq:mfg-system}), which is
\begin{definition}
\label{def:solution}
A quadruple $(u,\lambda,m,w)$ is a  variational solution at the endpoint case of
\eqref{eq:mfg-system} if $u\in W^{2,\qc}(\T^d)$ and $|Du|^\gamma\in L^{\qc}(\T^d)$ with $\int_{\T^d}u\,d x=0$ and  $m^{(\bar\alpha+1)/2}\in H^1(\T^d)$, $m\ge0$ with $\int_{\T^d}m\,d x=1$, which
 satisfies
 \begin{align}
 \left\{\begin{array}{ll}
  -\Delta u+|Du|^\gamma+\lambda=m^{\bar\alpha}+V
 \quad&\text{a.e. on }\T^d\\ \Delta m-\nabla\cdot w=0\quad&\text{in }\mathcal D'(\T^d),\\
  w=-mD_pH(Du) \text{ a.e.}, 
 \label{eq:solution-w}
 \end{array}
 \right.
\end{align}
where $d\geq4$, $\bar\alpha=\frac{\gp}{d-2-\gp}$, $\gp=\frac{\gamma}{\gamma-1}<d-2$ and  $V\in C^2(\T^d)$.  Moreover, $(m,w)\in\mathcal A$ is a minimizer to the following objective
\begin{equation}\label{eq:primal-action}
 \J(m,w):=\int_{\T^d}\left[
 \mathscr L(m,w)+\frac{m^{\bar\alpha+1}}{\bar\alpha+1}+Vm
 \right]\,d x,
\end{equation}
where 
\begin{equation}\label{eq:admissible-class}
 \A:=\left\{(m,w)\in L^{\bar\alpha+1}(\T^d)
 \times L^{r_0}(\T^d;\R^d):
 \begin{array}{l}
 m\ge0,\quad \displaystyle\int_{\T^d}m\,d x=1,\\
 \Delta m-\nabla\cdot w=0\quad\text{in }\mathcal D'(\T^d)
 \end{array}\right\}.
\end{equation}
Here, \(r_0:=\frac{(\bar\alpha+1)\gamma'}{\bar\alpha+\gamma'}\) and $\mathcal D'(\mathbb T^d)$ denotes the space of periodic
distributions, namely the continuous dual of
$C^\infty(\mathbb T^d)$.
\end{definition}

We shall employ the variational method to prove the existence of the variational solution to system (\ref{eq:mfg-system}).  First of all, we show that  $\min\limits_{(m,w)\in\mathcal A}\mathcal J(m,w)$ is finite, which is
\begin{lemma} 
\label{prop:minimizer}
Problem $\inf\limits_{(m,w)\in\A}\J(m,w)$ admits a unique minimizer $(m,w)$, where admissible set $\mathcal A$ and $\J$ are defined in (\ref{eq:admissible-class}) and (\ref{eq:primal-action}), respectively.
\end{lemma}

\begin{proof}
Using $(1,0)$ as a test, then we have $\min\limits_{(m,w)\in\mathcal A}\J(m,w)<+\infty$. Moreover, let $(m_j,w_j)\in\mathcal A$ be a minimizing
sequence.  Noting that  $V$ is bounded in $\T^d$ and $\int_{\T^d} m_j\,dx=1$, the convexity of $\mathcal J$ implies
\begin{equation}\label{eq:energy-coercivity}
 \sup_j\int_{\T^d}m_j^{\alpha+1}\,d x<\infty,
 \qquad
 \sup_j\int_{\T^d}\frac{|w_j|^{\gp}}{m_j^{\gp-1}}\,d x<\infty.
\end{equation}
By H\"older's inequality and $r_0$ shown in Definition \ref{def:solution}, we have
\begin{align}
 \int_{\T^d}|w_j|^{r_0}\,d x
 &\le
 \left(\int_{\T^d}\frac{|w_j|^{\gp}}{m_j^{\gp-1}}\,d x
 \right)^{r_0/\gp}
 \left(\int_{\T^d}m_j^{\alpha+1}\,d x
 \right)^{(\gp-r_0)/\gp},
 \label{eq:w-r0}
\end{align}
thanks to  $\frac{r_0(\gp-1)}{\gp-r_0}=\bar\alpha+1.$
The standard estimates of Fokker-Planck equations together with
$\int m_j=1$  yield 
\begin{equation}\label{eq:m-W1r0}
 \|m_j\|_{W^{1,r_0}(\T^d)}\le C(1+\|w_j\|_{L^{r_0}(\T^d)}).
\end{equation}
Consequently, after extraction,
\begin{equation}\label{eq:minimizer-convergences}
 m_j\rightharpoonup m\ \text{in }L^{\alpha+1}(\T^d),
 \qquad m_j\to m\ \text{in }L^1(\T^d),
 \qquad w_j\rightharpoonup w\ \text{in }L^{r_0}(\T^d).
\end{equation}
The constraint passes to the limit and $(m,w)\in\A$.  Invoking the convexity and weakly
lower semicontinuous property of $\J$, we find    $(m,w)$ minimizes $\J$.  Moreover, uniqueness of the minimizer is ensured by the convexity of $\J.$
\end{proof}

Next, we discuss the regularity of the minimizer $(m,w)$ obtained in Lemma \ref{prop:minimizer}.  Firs



\begin{lemma} 
\label{prop:translation}
Let $(m,w)$ be the minimizer from Proposition~\ref{prop:minimizer}.  Then
\begin{equation}\label{eq:translation-result}
 \phi:=m^{(\bar\alpha+1)/2}\in H^1(\T^d)
\end{equation}
and
\begin{equation}\label{eq:translation-H1-bound}
 \|D\phi\|_{L^2(\T^d)}^2
 \le C_{\bar\alpha}\|D^2V\|_{L^\infty(\T^d)}.
\end{equation}
Moreover, the following properties hold
\begin{equation}\label{eq:m-critical-integrability}
 m\in L^{\bar\alpha\qc}(\T^d),
 \qquad m^{\bar\alpha}\in L^{\qc}(\T^d).
\end{equation}
\end{lemma}

\begin{proof}
For $h\in\R^d$, we define
\[
 m_h(x):=m(x+h),\qquad w_h(x):=w(x+h)
\]
and consider the periodic extension.
 Letting  $(m_h,w_h)\in\mathcal A$ and
$(m_{-h},w_{-h})\in\mathcal A$, we define the midpoint as $ (\bar m_h,\bar w_h):=\frac12(m_h+m_{-h},w_h+w_{-h})$.  
Decompose
\[
 \J(m,w)=\J_0(m,w)+\int_{\T^d}Vm\,d x,
 \qquad
 \J_0(m,w):=\int_{\T^d}\left[
 \mathscr L(m,w)+\frac{m^{\bar\alpha+1}}{\bar\alpha+1}\right]\,d x.
\]
By the translation invariance property, we obtain  $ \J_0(m_h,w_h)=\J_0(m_{-h},w_{-h})=\J_0(m,w).$
Then, the minimality of $(m,w)$ implies
\begin{equation}\label{eq:minimality-midpoint}
 \J_0(m,w)-\J_0(\bar m_h,\bar w_h)
 \le\int_{\T^d}V(\bar m_h-m)\,d x.
\end{equation}
Using a change of variables, we further obtain 
\begin{align}
 \int_{\T^d}V(\bar m_h-m)\,d x
 &=\int_{\T^d}m(x)
 \left(\frac{V(x+h)+V(x-h)}2-V(x)\right)\,d x
 \notag\\
 &\le\frac12\|D^2V\|_{L^\infty}|h|^2
 \int_{\T^d}m\,d x
 \le C\|D^2V\|_{L^\infty}|h|^2.
 \label{eq:potential-second-difference}
\end{align}
On the other hand, the uniformly convexity indicates
\begin{align}
 \J_0(m,w)-\J_0(\bar m_h,\bar w_h)
 &=\frac{\J_0(m_h,w_h)+\J_0(m_{-h},w_{-h})}{2}
 -\J_0(\bar m_h,\bar w_h)
 \notag\\
 &\ge c_{\bar\alpha}\int_{\T^d}
 \left|m_h^{(\bar\alpha+1)/2}-m_{-h}^{(\bar\alpha+1)/2}\right|^2\,d x,
 \label{eq:Jensen-deficit}
\end{align}
where $c_{\bar\alpha}>0$ is a constant.
Combining \eqref{eq:minimality-midpoint}--\eqref{eq:Jensen-deficit} yields
\begin{equation}\label{eq:phi-two-h}
 \|\phi(\,\cdot+h)-\phi(\,\cdot-h)\|_{L^2}
 \le C|h|.
\end{equation}
 The difference-quotient characterization of $H^1(\T^d)$ then proves
\eqref{eq:translation-result} and \eqref{eq:translation-H1-bound}.

To show (\ref{eq:m-critical-integrability}), in light of  $d>2$, we have from Lemma ~\ref{prop:translation} and Sobolev's inequality that  $\phi=m^{(\bar\alpha+1)/2}
 \in L^{2d/(d-2)}(\T^d).$
Therefore, $m\in L^{(\bar\alpha+1)d/(d-2)}(\T^d).$
By using \eqref{eq:critical-identities}, we have $\frac{(\bar\alpha+1)d}{d-2}=\bar\alpha q_c$, which proves
\eqref{eq:m-critical-integrability}.
\end{proof}

 \subsection{Existence of the value function}
In this subsection, after obtaining $(m,w)$ by minimization problem $\min_{(m,w)\in\mathcal A}\J$, we show the existence of the value function $u$ solving Hamilton--Jacobi equations.  To begin with, we prove the following preliminary lemma.
 \begin{lemma} 
\label{lem:strong-equi}
Let $1<p<\infty$ and $f_k\to f$  in $L^p(\T^d)$.  Then
$\{f_k:k\in\mathbb N\}\cup\{f\}$ is uniformly equi-integrable in $L^p$.
 
\end{lemma}

\begin{proof}
Fix $\varepsilon>0$.  Choose $k_0$ such that
$\|f_k-f\|_{L^p}^p<\varepsilon/2^p$ for $k\ge k_0$.  Invoking  the absolute
continuity of the integral, we have there exists $s_0>0$ such that for $|E|\le s_0$,
 $\int_E|f|^p\,d x<\frac\varepsilon{2^p}.$
In addition, for $k\ge k_0$ and $|E|\le s_0$,
\begin{align}\label{togetherwithestimate}
 \int_E|f_k|^p\,d x
 \le 2^{p-1}\|f_k-f\|_{L^p}^p
     +2^{p-1}\int_E|f|^p\,d x
 <\varepsilon.
\end{align}
For each $1\le j<k_0$, the absolute continuity of the integral of
$|f_j|^p$ yields $s_j>0$ such that  for  $|E|\le s_j,$ $\int_E |f_j|^p\,dx<\varepsilon.$
Choosing $\tilde s_0:=\min\{s_0,s_1,\ldots,s_{k_0-1}\}>0,$
we obtain for $|E|\le \tilde s_0,$ $\int_{E}|f_k|^p<\varepsilon$ for every $1\le j<k_0$. Together with the
estimate (\ref{togetherwithestimate}), this proves that
$(f_k)_k$ is uniformly equi-integrable in $L^p(\mathbb T^d)$.
 
\end{proof}

Now, we are ready to show the existence of the value function $u$, which is
\begin{lemma}
\label{prop:fixed-HJ}
Let $f\in L^{\qc}(\T^d)$ be bounded from below.  There exist
$u\in W^{2,\qc}(\T^d)$ and $\lambda\in\R$ such that
\begin{equation}\label{eq:fixed-HJ}
 -\Delta u+|Du|^\gamma+\lambda=f\quad\text{a.e.},
 \qquad \int_{\T^d}u\,d x=0,
\end{equation}
and $|Du|^\gamma\in L^{\qc}(\T^d)$.
\end{lemma}

\begin{proof}
Choose periodic mollifications $f_k\in C^\infty(\T^d)$ such that $f_k\to f$  in $L^{\qc}(\T^d)$ and $\inf_{\T^d}f_k\ge c_0$,
where $c_0$ is a fixed lower bound for $f$.   In light of \cite[Theorem~1]{CesaroniCirant2019}, we find for every $k$, there exist a
unique ergodic constant $\lambda_k\in\mathbb R$ and a periodic classical
solution $u_k$, unique up to an additive constant, such that
\begin{align}\label{eq:smooth-ergodic-HJ}
-\Delta u_k+|Du_k|^\gamma+\lambda_k=f_k
\qquad\text{in }\mathbb T^d.
\end{align}
We normalize $u_k$ by letting $\int_{\mathbb T^d}u_k\,dx=0.$  At a minimum point of $u_k$, one has $Du_k=0$ and
$\Delta u_k\ge0$; hence $ \lambda_k=f_k+\Delta u_k\ge c_0.$
In addition, integration of \eqref{eq:smooth-ergodic-HJ} gives
\begin{equation}\label{eq:lambda-upper}
 \lambda_k=\int_{\T^d}f_k\,d x
 -\int_{\T^d}|Du_k|^\gamma\,d x
 \le\int_{\T^d}f_k\,d x.
\end{equation}
Thus $(\lambda_k)_k$ is bounded.  Thanks to Lemma~\ref{lem:strong-equi}, the family
$\{f_k-\lambda_k\}$ is bounded and uniformly equi-integrable in
$L^{\qc}$.  Theorem~\ref{thm:main} therefore yields
\begin{equation}\label{eq:HJ-uniform-bound}
 \sup_k\left(
 \|D^2u_k\|_{L^{\qc}(\T^d)}
 +\bigl\||Du_k|^\gamma\bigr\|_{L^{\qc}(\T^d)}
 \right)<+\infty.
\end{equation}
Poincar\'e inequality and the normalization imply a uniform
$W^{2,\qc}$ bound of $(u_k)_k$.  Moreover, by Sobolev embedding theorem, one has   $u_k\rightharpoonup u$ in $W^{2,\qc}(\T^d)$, $u_k\to u$ in $W^{1,s}(\T^d)$ for every $s<\gamma\qc$ and $\lambda_k\to\lambda$ in $\R$.  Thus, we find $Du_k\to Du$ a.e.
  In addition, the sequence $|Du_k|^\gamma$ is bounded in $L^{\qc}(\T^d)$ by
\eqref{eq:HJ-uniform-bound}; since $\qc>1$, we have $|Du_k|^\gamma\rightharpoonup |Du|^\gamma$
in $L^{\qc}(\T^d).$  Passing to the limit in \eqref{eq:smooth-ergodic-HJ} proves
\eqref{eq:fixed-HJ}  and finishes the proof of this lemma.
\end{proof}

\subsection{Linearized problem and duality argument}
\label{sec:duality}
This subsection is devoted the duality argument of the value function $u$ for showing the variational characterization between $u$ and $w$. To begin with, we introduce the linearized problem of $\min\limits_{(m,w)\in\A}\J$.

Set
\begin{equation}\label{eq:linear-spaces}
 s:=\qc'=\frac d{d-\gp},
 \qquad
 r:=\frac d{d+1-\gp},
\end{equation}
so that $1<r<\gp$ and $s=dr/(d-r)$.  Define
\begin{equation}\label{eq:linear-admissible-class}
 \A_{\mathrm{lin}}:=\left\{(m,w)\in L^s(\T^d)
 \times L^r(\T^d;\R^d):
 \begin{array}{l}
 m\ge0,\quad \displaystyle\int_{\T^d}m\,d x=1,\\
 \Delta m-\nabla\cdot w=0\quad\text{in }\mathcal D'(\T^d)
 \end{array}\right\},
\end{equation}
where $\bar\alpha+1>s$ and $r_0>r$ given in (\ref{eq:admissible-class}); hence
$\A\subset\A_{\mathrm{lin}}$.

For $f\in L^{\qc}(\T^d)$, define the linearized functional of $\J$ given in (\ref{eq:primal-action}) as 
\begin{equation}\label{eq:linear-functional}
 \K_f(m,w):=\int_{\T^d}\bigl[\mathscr L(m,w)+fm\bigr]\,d x,
 \qquad (m,w)\in\A_{\mathrm{lin}}.
\end{equation}
We next establish the duality argument for the minimization problem $\inf\limits_{(m,w)\in\mathcal A_{\mathrm{lin}}}\K_f(m,w),$ which is
\begin{proposition} [e.g. Proposition 3.4 in \cite{CesaroniCirant2019GroundStates}] 
\label{prop:linear-duality}
Let $f\in L^{\qc}(\T^d)$ be bounded from below.  Then
\begin{equation}\label{eq:linear-duality}
 \inf_{(m,w)\in\A_{\mathrm{lin}}}\K_f(m,w)
 =\sup\left\{\mu\in\R:
 \begin{array}{l}
 \text{there exists }v\in W^{2,\qc}(\T^d),\ \int_{\T^d}v=0,\\[-1mm]
 -\Delta v+H(Dv)+\mu\le f\quad\text{a.e.}
 \end{array}\right\}.
\end{equation}
Moreover, both extrema in \eqref{eq:linear-duality} are attained.
More precisely, every normalized solution $(u,\lambda)$ of
\eqref{eq:fixed-HJ} is a maximizer of the dual problem, and the primal
problem admits a minimizer. For every primal minimizer $(m,w)$, one has
\begin{equation}\label{eq:feedback-equality}
 w=-mD_pH(Du)
 \qquad\text{a.e. in }\T^d.
\end{equation}

Suppose that $(u,\lambda)$ satisfies \eqref{eq:fixed-HJ}, i.e. it is the optimizer to  the
right-hand side of (\ref{eq:linear-duality}).  Also, assume that the left-hand side of (\ref{eq:linear-duality}) is attained by $(m,w)$.  Then, we have
\begin{equation}
 w=-mD_pH(Du)\quad\text{a.e. in }\T^d.
\end{equation}
\end{proposition}
\begin{proof}
We divide the proof into five steps.

\medskip
\noindent
\textbf{Step 1: weak duality.}
Let $(m,w)\in\A_{\mathrm{lin}}$ and
$(v,\mu)$ be dual pair, namely
\begin{equation}\label{eq:dual-admissibility-detailed}
 v\in W^{2,\qc}(\T^d),
 \qquad
 \int_{\T^d}v\,dx=0,
 \qquad
 -\Delta v+H(Dv)+\mu\le f
 \quad\text{a.e. in }\T^d.
\end{equation}
Invoking (\ref{dual_representation}) and taking $p=Dv$, we get
\begin{equation}\label{eq:Fenchel-pointwise}
 \mathscr L(m,w)
 \ge
 -w\cdot Dv-mH(Dv)
 \qquad\text{a.e. in }\T^d.
\end{equation}
We next verify  the  constraint in the distributional sense. In light of \eqref{eq:linear-spaces}, Sobolev embedding yields  $W^{2,\qc}(\T^d)
 \hookrightarrow
 W^{1,\gamma\qc}(\T^d)
 =
 W^{1,r'}(\T^d)$  with $r'=\frac{r}{r-1}.$
Consequently, $\Delta v\in L^{\qc}(\T^d)$ and  $Dv\in L^{r'}(\T^d).$
Since $m\in L^s=L^{\qc'}(\T^d)$ and $w\in L^r(\T^d)$ with $s$ given in (\ref{eq:linear-spaces}), H\"older's inequality yields
\[
 m\Delta v\in L^1(\T^d),
 \qquad
 w\cdot Dv\in L^1(\T^d).
\]

Choose $v_j\in C^\infty(\T^d)$ such that
$v_j\to v$ in $W^{2,\qc}(\T^d)$. For every $j$, the
constraint equation gives
\[
 \int_{\T^d}
 \bigl(m\Delta v_j+w\cdot Dv_j\bigr)\,dx=0.
\]
Moreover, $ \left|
 \int_{\T^d}m\Delta(v_j-v)\,dx
 \right|
 \le
 \|m\|_{L^s(\T^d)}
 \|\Delta(v_j-v)\|_{L^{\qc}(\T^d)}
 \rightarrow0$
and $\left|
 \int_{\T^d}w\cdot D(v_j-v)\,dx
 \right|
 \le
 \|w\|_{L^r(\T^d)}
 \|D(v_j-v)\|_{L^{r'}(\T^d)}
 \rightarrow0.$
Therefore,
\begin{equation}\label{eq:constraint-test}
 \int_{\T^d}
 \bigl(m\Delta v+w\cdot Dv\bigr)\,dx=0.
\end{equation}
Combining \eqref{eq:Fenchel-pointwise} and
\eqref{eq:constraint-test}, we obtain
\begin{align}
 \K_f(m,w)
 &=
 \int_{\T^d}
 \bigl[\mathscr L(m,w)+fm\bigr]\,dx
 \notag\\
 &\ge
 \int_{\T^d}
 \bigl[-w\cdot Dv-mH(Dv)+fm\bigr]\,dx
 \notag\\
 &=
 \int_{\T^d}
 m\bigl[\Delta v-H(Dv)+f\bigr]\,dx.
 \label{eq:weak-duality-computation}
\end{align}
By using \eqref{eq:dual-admissibility-detailed}, we have
\[
 \Delta v-H(Dv)+f\ge\mu
 \qquad\text{a.e. in }\T^d.
\]
Since $m\ge0$ and $\int_{\T^d}m\,dx=1$, it follows that
\begin{equation}\label{eq:weak-linear-duality}
 \K_f(m,w)\ge
 \mu\int_{\T^d}m\,dx=\mu.
\end{equation}
Taking first the infimum over $(m,w)\in\A_{\mathrm{lin}}$ and then the
supremum over all dual pair $(v,\mu)$ proves weak duality:
\begin{equation}\label{eq:weak-duality-values}
 \sup_{(v,\mu)\in\mathcal D_f}\mu
 \le
 \inf_{(m,w)\in\A_{\mathrm{lin}}}\K_f(m,w),
\end{equation}
where 
\begin{equation}\label{eq:dual-admissible-set}
 \mathcal D_f
 :=
 \left\{
 (v,\mu)\in W^{2,q_c}(\T^d)\times\mathbb R:
 \int_{\T^d}v\,dx=0,\quad
 -\Delta v+H(Dv)+\mu\le f\ \text{a.e.}
 \right\}.
\end{equation}

\medskip
\noindent
\textbf{Step 2: smooth approximation of primal optimizers.}
Let $f_k$ and $(u_k,\lambda_k)$ be the approximations constructed in
Lemma~\ref{prop:fixed-HJ}. Thus, $f_k\to f$
 in $L^{\qc}(\T^d)$, $f_k\ge c_0$, $\lambda_k\to\lambda$,
and
\begin{equation}\label{eq:smooth-ergodic-HJ-duality}
 -\Delta u_k+H(Du_k)+\lambda_k=f_k,
 \qquad
 \int_{\T^d}u_k\,dx=0.
\end{equation}
Set $b_k:=D_pH(Du_k)$ and
let $m_k$ satisfy
\[ -\Delta m_k-\nabla\cdot(m_kb_k)=0,
 \qquad
 m_k>0,
 \qquad
 \int_{\T^d}m_k\,dx=1.\]
 Define $w_k:=-m_kb_k,$ then $\Delta m_k-\nabla\cdot w_k
 =
 \Delta m_k+\nabla\cdot(m_kb_k)=0,$
so $(m_k,w_k)\in\A_{\mathrm{lin}}$.  Since $b_k=D_pH(Du_k)$, we have from Fenchel duality that $L(b_k)
 =
 b_k\cdot Du_k-H(Du_k).$
Since $-w_k/m_k=b_k$, we have
\begin{align}
 \mathscr L(m_k,w_k)
 &=
 m_kL\left(-\frac{w_k}{m_k}\right)=
 m_kL(b_k)=
 m_k\bigl[b_k\cdot Du_k-H(Du_k)\bigr],
 \label{eq:perspective-equality-k}
\end{align}
where $\mathscr L$ is given by (\ref{eq:perspective}).

Testing the constraint for $(m_k,w_k)$ against $u_k$ gives
\[
 \int_{\T^d}
 \bigl(m_k\Delta u_k+w_k\cdot Du_k\bigr)\,dx=0.
\]
Since $w_k=-m_kb_k$, this becomes
\begin{equation}\label{eq:constraint-k-rewritten}
 \int_{\T^d}m_kb_k\cdot Du_k\,dx
 =
 \int_{\T^d}m_k\Delta u_k\,dx.
\end{equation}
Using \eqref{eq:perspective-equality-k},
\eqref{eq:constraint-k-rewritten} and
\eqref{eq:smooth-ergodic-HJ-duality}, we compute
\begin{align}
 \K_{f_k}(m_k,w_k)
 &=
 \int_{\T^d}
 \bigl[
 \mathscr L(m_k,w_k)+f_km_k
 \bigr]\,dx
 \notag\\
 &=
 \int_{\T^d}
 m_k
 \bigl[
 b_k\cdot Du_k-H(Du_k)+f_k
 \bigr]\,dx
 \notag\\
 &=
 \int_{\T^d}
 m_k
 \bigl[
 \Delta u_k-H(Du_k)+f_k
 \bigr]\,dx
 \notag\\
 &=
 \int_{\T^d}\lambda_km_k\,dx
 =
 \lambda_k.
 \label{eq:smooth-primal-value}
\end{align}
For each $k$, define the primal and dual optimal values by
\[
 P_k
 :=
 \inf_{(m,w)\in\A_{\mathrm{lin}}}
 \K_{f_k}(m,w),
 \qquad
 D_k
 :=
 \sup_{(v,\mu)\in\mathcal D_{f_k}}\mu.
\]
Combining \eqref{eq:weak-duality-values} with \eqref{eq:smooth-primal-value}, we further obtain
\[
 \lambda_k
 \le D_k
 \le P_k
 \le \K_{f_k}(m_k,w_k)
 =\lambda_k.
\]
Consequently, $D_k=P_k=\lambda_k $
and we have $(m_k,w_k)$ is a minimizer of the primal problem $\inf\K_f(m,w)$ associated with
$f_k$, while $(u_k,\lambda_k)$ is a maximizer of the
dual problem $\sup \mu$.

\medskip
\noindent
\textbf{Step 3: regularity estimates and compactness of the approximating sequence.}
Noting that $H(p)=|p|^\gamma$, we have  $\mathscr L(m,w)$ becomes
\[
 \mathscr L(m,w)
 =
 c_\gamma\frac{|w|^{\gp}}{m^{\gp-1}}
 \qquad\text{on }\{m>0\},
\]
where $c_\gamma=(\gamma-1)\gamma^{-\gp}>0.$
Thus, by \eqref{eq:smooth-primal-value},
\begin{align}
 c_\gamma
 \int_{\T^d}
 \frac{|w_k|^{\gp}}{m_k^{\gp-1}}\,dx
 &=
 \K_0(m_k,w_k)
 =
 \lambda_k-\int_{\T^d}f_km_k\,dx.
 \label{eq:energy-identity-k}
\end{align}
Since $f_k\ge c_0$ and $\int_{\T^d}m_k\,dx=1$,
\[
 \int_{\T^d}f_km_k\,dx\ge c_0.
\]
Moreover, since $(\lambda_k)_k$ is bounded, we have from (\ref{eq:energy-identity-k}) that
\begin{equation}\label{eq:linear-energy-bound}
 E_k:=
 \int_{\T^d}
 \frac{|w_k|^{\gp}}{m_k^{\gp-1}}\,dx
 \le C,
\end{equation}
where constant $C>0$ independent of $k$.

We next derive the \(L^r\)-bound for $w_k$.   Indeed, applying H\"older's inequality with exponents $\frac{\gp}{r}$ and $\frac{\gp}{\gp-r}$,
we obtain
\begin{align}
 \int_{\T^d}|w_k|^r\,dx
 &\le
 E_k^{r/\gp}
 \left(
 \int_{\T^d}
 m_k^{\,r(\gp-1)/(\gp-r)}\,dx
 \right)^{(\gp-r)/\gp},
 \label{eq:Holder-flux-detailed}
\end{align}
where
\begin{equation}\label{eq:rs-critical-identities}
 \frac{r(\gp-1)}{\gp-r}=s,
 \qquad
 \frac{dr}{d-r}=s.
\end{equation}
Taking the \(r\)-th root in
\eqref{eq:Holder-flux-detailed} and using
\eqref{eq:linear-energy-bound}, we find
\begin{equation}\label{eq:linear-wr-bound}
 \|w_k\|_{L^r}
 \le
 C\|m_k\|_{L^s}^{(\gp-1)/\gp}.
\end{equation}

The standard estimate for $\Delta m_k=\operatorname{div}w_k$
 with $\int_{\T^d}m_k\,dx=1$ gives
\begin{equation}\label{eq:periodic-m-estimate}
 \|m_k\|_{W^{1,r}}
 \le
 C\bigl(1+\|w_k\|_{L^r}\bigr).
\end{equation}
Since \(W^{1,r}(\T^d)\hookrightarrow L^s(\T^d)\) with $r$ and $s$ given in 
\eqref{eq:rs-critical-identities}, we have
\begin{equation}\label{eq:linear-ms-bound-pre}
 \|m_k\|_{L^s}
 \le
 C\bigl(1+\|w_k\|_{L^r}\bigr).
\end{equation}
Collecting \eqref{eq:linear-wr-bound} and
\eqref{eq:linear-ms-bound-pre}, we obtain $\sup_k\|m_k\|_{L^s}<\infty.$
Combining \eqref{eq:linear-wr-bound} with
\eqref{eq:periodic-m-estimate}, we conclude that
\begin{equation}\label{eq:linear-compactness-bound}
 \sup_k
 \left(
 \|m_k\|_{L^s}
 +\|m_k\|_{W^{1,r}}
 +\|w_k\|_{L^r}
 \right)<\infty.
\end{equation}
In light of (\ref{eq:linear-compactness-bound}), after passing to a
subsequence, there exist \(m_0\) and \(w_0\) such that
\begin{align}
 m_k&\rightharpoonup m_0
 &&\text{weakly in }W^{1,r}(\T^d),
 \label{eq:mk-weak-W1r}\\
 m_k&\rightharpoonup m_0
 &&\text{weakly in }L^s(\T^d),
 \label{eq:mk-weak-Ls}\\
 w_k&\rightharpoonup w_0
 &&\text{weakly in }L^r(\T^d).
 \label{eq:wk-weak-Lr}
\end{align}
Since \(W^{1,r}(\T^d)\) is compactly embedded into \(L^1(\T^d)\),
we obtain
\begin{equation}\label{eq:mk-strong-L1}
 m_k\to m_0
 \qquad\text{ in }L^1(\T^d)
 \quad\text{and a.e. in }\T^d.
\end{equation}
In particular,
\[
 m_0\ge0,
 \qquad
 \int_{\T^d}m_0\,dx
 =
 \lim_{k\to\infty}\int_{\T^d}m_k\,dx=1.
\]

We next verify the constraint equation.  For every \(\varphi\in C^\infty(\T^d)\), we have
\begin{align}\label{smooth_constraint}
 \int_{\T^d}
 \bigl(m_k\Delta\varphi+w_k\cdot D\varphi\bigr)\,dx=0.
\end{align}
Passing to the limit by
\eqref{eq:mk-weak-Ls} and \eqref{eq:wk-weak-Lr} in (\ref{smooth_constraint}), we obtain
\[
 \int_{\T^d}
 \bigl(m_0\Delta\varphi+w_0\cdot D\varphi\bigr)\,dx=0.
\]
Thus, $\Delta m_0-\nabla\cdot w_0=0$ in $\mathcal D'(\T^d).$  In summary, we obtain  $(m_0,w_0)\in\A_{\mathrm{lin}}.$

\medskip
\noindent
\textbf{Step 4: passage to the limit and strong duality.}
First, since \(f_k\to f\)  in \(L^{\qc}\) and
\((m_k)_k\) is bounded in \(L^{\qc'}(\T^d)\), we have
\begin{align}
 \left|
 \int_{\T^d}(f_k-f)m_k\,dx
 \right|
 &\le
 \|f_k-f\|_{L^{\qc}(\T^d)}
 \|m_k\|_{L^s(\T^d)}
 \rightarrow0.
 \label{eq:linear-source-error}
\end{align}
Moreover, by the weak convergence of \(m_k\) in \(L^s\), we have from \eqref{eq:linear-source-error}  that
\begin{equation}\label{eq:fixed-source-weak-limit}
 \int_{\T^d}fm_k\,dx
 \longrightarrow
 \int_{\T^d}fm_0\,dx.
\end{equation}
By using the weakly lower
semi-continuous property of $\K_0$ given in (\ref{eq:linear-functional}) with $f=0$, we further obtain
\begin{equation}\label{eq:perspective-lsc}
 \K_0(m_0,w_0)
 \le
 \liminf_{k\to\infty}\K_0(m_k,w_k).
\end{equation}
Combining \eqref{eq:fixed-source-weak-limit} and
\eqref{eq:perspective-lsc}, we obtain
\begin{align}
 \K_f(m_0,w_0)
 &=
 \K_0(m_0,w_0)
 +\int_{\T^d}fm_0\,dx
 \notag\\
 &\le
 \liminf_{k\to\infty}
 \left[
 \K_0(m_k,w_k)
 +\int_{\T^d}fm_k\,dx
 \right]
 \notag\\
 &=
 \liminf_{k\to\infty}
 \left[
 \K_{f_k}(m_k,w_k)
 -
 \int_{\T^d}(f_k-f)m_k\,dx
 \right]
 \notag\\
 &=
 \liminf_{k\to\infty}\lambda_k
 =
 \lambda.
 \label{eq:critical-primal-upper-bound}
\end{align}
Since \((m_0,w_0)\in\A_{\mathrm{lin}}\), one has

 \begin{align}\label{combine1_inequality}\inf_{(m,w)\in\A_{\mathrm{lin}}}\K_f(m,w)
 \le
 \K_f(m_0,w_0)
 \le\lambda.
\end{align}
On the other hand, as shown in Lemma \ref{prop:fixed-HJ}, we get \((u,\lambda)\) solves
\[
 -\Delta u+H(Du)+\lambda=f,
\]
so $(u,\lambda)\in\mathcal D_f$ given in (\ref{eq:dual-admissible-set}). Weak duality (\ref{eq:weak-duality-values}) therefore gives
\begin{align}\label{combine2_inequality}
 \lambda
 \le
 \sup_{(v,\mu)\in\mathcal D_f}\mu
 \le
 \inf_{(m,w)\in\A_{\mathrm{lin}}}\K_f(m,w).
\end{align}
Combining  (\ref{combine1_inequality}) with  (\ref{combine2_inequality}) yields
\begin{equation}\label{eq:all-dual-values-equal}
 \lambda
 =
 \sup_{(v,\mu)\in\mathcal D_f}\mu
 =
 \inf_{(m,w)\in\A_{\mathrm{lin}}}\K_f(m,w)
 =
 \K_f(m_0,w_0).
\end{equation}
Thus, \((u,\lambda)\) is a dual optimizer and
\((m_0,w_0)\) is a primal optimizer.

\medskip
\noindent
\textbf{Step 5:  relation between $m$ and $w$.}
Let \((m,w)\) be any primal minimizer of $\min\limits_{(m,w)\in\mathcal A}\J$. By using
\eqref{eq:all-dual-values-equal}, one has $\K_f(m,w)=\lambda.$
Since \(\int_{\T^d}m\,dx=1\) and $f=-\Delta u+H(Du)+\lambda,$
we have
\begin{align}
0
&=
\K_f(m,w)-\lambda
=
\int_{\T^d}
\bigl[
\mathscr L(m,w)+fm-\lambda m
\bigr]\,dx
\notag\\
&=
\int_{\T^d}
\bigl[
\mathscr L(m,w)-m\Delta u+mH(Du)
\bigr]\,dx,
\label{eq:duality-gap-before-constraint}
\end{align}
where $\mathscr L$ is given in (\ref{eq:perspective}).
Testing the constraint equation
against \(u\) gives
\[
 \int_{\T^d}
 \bigl(m\Delta u+w\cdot Du\bigr)\,dx=0.
\]
Substituting this identity into
\eqref{eq:duality-gap-before-constraint}, we obtain
\begin{equation}\label{eq:integrated-Fenchel-gap}
 \int_{\T^d}
 \bigl[
 \mathscr L(m,w)+w\cdot Du+mH(Du)
 \bigr]\,dx=0.
\end{equation}
Noting that the pointwise Fenchel inequality gives
 $\mathscr L(m,w)+w\cdot Du+mH(Du)\ge0$
a.e. in $\T^d,$
where
\[
 \mathscr L(m,w)\in L^1,
 \qquad
 w\cdot Du\in L^1,
 \qquad
 mH(Du)\in L^1.
\]
It then follows from (\ref{eq:integrated-Fenchel-gap}) that
\begin{equation}\label{eq:pointwise-Fenchel-equality-detailed}
 \mathscr L(m,w)
 =
 -w\cdot Du-mH(Du)
 \qquad\text{a.e. on }\T^d.
\end{equation}
Recall that on the set \(\{m>0\}\), $\mathscr L(m,w)
 =
 \sup_{p\in\mathbb R^d}
 \bigl\{-w\cdot p-mH(p)\bigr\}.$
\eqref{eq:pointwise-Fenchel-equality-detailed} yields that the
supremum is attained at \(p=Du\). Since \(H\) is differentiable and
strictly convex, the first-order optimality condition is
\[
 -w-mD_pH(Du)=0,
\]
which implies
\begin{equation}\label{eq:feedback-on-positive-set}
 w=-mD_pH(Du)
 \qquad\text{a.e. on }\{m>0\}.
\end{equation}

On the set \(\{m=0\}\),  we have from  definition (\ref{eq:perspective}) that 
\[
 \mathscr L(0,w)
 =
 \begin{cases}
 0,&w=0,\\
 +\infty,&w\ne0.
 \end{cases}
\]
Since $\int_{(m,w)\in\A}\J$ is finite, we have \(\mathscr L(m,w)<\infty\) a.e. Therefore, $w=0$
a.e. on $\{m=0\}.$
Together with \eqref{eq:feedback-on-positive-set}, this proves
\begin{equation}
 w=-mD_pH(Du)
 \qquad\text{a.e. on }\T^d.
\end{equation}
    Noting that  \(H(p)=|p|^\gamma\), we obtain the following  relation 
\[
 w=-\gamma m|Du|^{\gamma-2}Du.
\]
This completes the proof.
\end{proof}
We next show that the minimizer of $\J$ is also the minimizer of the linearized problem $\mathcal K_f,$ which is
\begin{proposition}
\label{prop:linearized-optimality}
Let $(m,w)$ minimize $\J$, and set $f:=m^{\bar\alpha}+V.$
Then $(m,w)$ minimizes $\K_f$ over $\A_{\mathrm{lin}}$, where $\A_{\mathrm{lin}}$ is given in (\ref{eq:linear-admissible-class}).
\end{proposition}

\begin{proof}
First of all, we fix a smooth pair
$(\widetilde m,\widetilde w)\in\A_{\mathrm{lin}}$ with
$\widetilde m>0$. It is straightforward to see $(\widetilde m,\widetilde w)\in\A$.  Define $(m_t,w_t):=(1-t)(m,w)+t(\widetilde m,\widetilde w)$, $0<t<1$.
Let
\[
 \K_0(m,w):=\int_{\T^d}\mathscr L(m,w)\,d x,~ \mathcal G(m):=
 \int_{\T^d}\left(\frac{m^{\bar\alpha+1}}{\bar\alpha+1}+Vm\right)\,d x,
\]
{ 
then we have 
\[\J(m,w)=\K_0(m,w)+\mathcal G(m).
\]
Since $(m_t,w_t)\in\A$, the minimality of $(m,w)$ yields
\begin{equation}\label{eq:minimality-linearized-path}
 0\le \J(m_t,w_t)-\J(m,w).
\end{equation}
 We  use the convexity of $\K_0$ to get
\[
 \K_0(m_t,w_t)
 \le (1-t)\K_0(m,w)
      +t\K_0(\widetilde m,\widetilde w).
\]
Combining this inequality with
\eqref{eq:minimality-linearized-path} and dividing by $t>0$, we obtain
\begin{equation}\label{eq:convex-linearized-inequality}
 0\le
 \K_0(\widetilde m,\widetilde w)-\K_0(m,w)
 +\frac{\mathcal G(m_t)-\mathcal G(m)}{t}.
\end{equation}
Noting that $\bar\alpha+1>1$, we have
\[
 \lim_{t\downarrow0}
 \frac{\mathcal G(m_t)-\mathcal G(m)}{t}
 =\int_{\T^d}(m^\alpha+V)(\widetilde m-m)\,d x.
\]
Thus, letting $t\downarrow0$ in
\eqref{eq:convex-linearized-inequality}, we conclude that
\[
 0\le
 \K_0(\widetilde m,\widetilde w)-\K_0(m,w)
 +\int_{\T^d}(m^\alpha+V)(\widetilde m-m)\,d x,
\]
which implies
\begin{align}\label{smooth_Kf_min}
\K_f(m,w)\le\K_f(\widetilde m,\widetilde w).
\end{align}
}

It remains to show that the inequality holds for  an arbitrary
$(\widetilde m,\widetilde w)\in\A_{\mathrm{lin}}$ with finite
$\K_f$.  Let $\rho_\varepsilon$ be a periodic mollifier and define
\[
 \widetilde m_\varepsilon
 :=\frac{\rho_\varepsilon*\widetilde m+\varepsilon}{1+\varepsilon},
 \qquad
 \widetilde w_\varepsilon
 :=\frac{\rho_\varepsilon*\widetilde w}{1+\varepsilon}.
\]
Then $(\widetilde m_\varepsilon,\widetilde w_\varepsilon)$ is smooth,
strictly positive, satisfying the linear constraint and $\int_{\T^d} \tilde m_{\varepsilon}\,dx=1$.
Convexity of $\K_0$ and Jensen's inequality indicate
\begin{align}\label{jesen_smooth_competitor}
\limsup_{\varepsilon\downarrow0}
 \int_{\T^d}\mathscr L(\widetilde m_\varepsilon,
 \widetilde w_\varepsilon)\,d x
 \le \int_{\T^d}\mathscr L(\widetilde m,\widetilde w)\,d x.
\end{align}
Moreover, $\widetilde m_\varepsilon\to\widetilde m$  in
$L^s(\T^d)$ with $s$ given in (\ref{eq:linear-spaces}).  Thanks to H\"older's inequality, one finds
\begin{align}\label{linear_convergence_competitor}
\int_{\T^d} f\widetilde m_\varepsilon\,dx\to\int_{\T^d} f\widetilde m\,dx.  
\end{align}
For every $\varepsilon>0$, the pair
$(\widetilde m_\varepsilon,\widetilde w_\varepsilon)$ is a smooth pair. Hence
\eqref{smooth_Kf_min}  yields $\K_f(m,w)
 \le
 \K_f(\widetilde m_\varepsilon,\widetilde w_\varepsilon).$
Taking the lower and upper limits as $\varepsilon\downarrow0$, we obtain
\[
\begin{aligned}
 \K_f(m,w)
 &\le
 \liminf_{\varepsilon\downarrow0}
 \K_f(\widetilde m_\varepsilon,\widetilde w_\varepsilon)\\
 &\le
 \limsup_{\varepsilon\downarrow0}
 \K_f(\widetilde m_\varepsilon,\widetilde w_\varepsilon)
 \le
 \K_f(\widetilde m,\widetilde w),
\end{aligned}
\]
where the last inequality holds due to (\ref{jesen_smooth_competitor}) and (\ref{linear_convergence_competitor}).
Since $(\widetilde m,\widetilde w)$ is arbitrary, we have $(m,w)$ minimizes
$\K_f$ over $\A_{\mathrm{lin}}$.
\end{proof}

\subsection{Existence for mean-field games with endpoint coupling exponents}
In this section, we finish the proof of Theorem \ref{thm:main2} for the existence of solutions to mean-field games systems with the endpoint coupling exponent.
\begin{proof}[Proof of Theorem~\ref{thm:main2}]
Let $(m,w)$ be the unique minimizer of $\J$ shown in 
Lemma~\ref{prop:minimizer}.  Then, Lemma~\ref{prop:translation} implies  $m^{(\bar\alpha+1)/2}\in H^1(\T^d)$ and
$m^{\bar\alpha}\in L^{\qc}(\T^d).$  In light of $m^\alpha\ge0$ and $V\in C^2(\T^d)$, the source
$f:=m^{\bar\alpha}+V\in L^{\qc}$, which is bounded from below.
Lemma~\ref{prop:fixed-HJ} therefore concludes that there is a normalized pair
$(u,\lambda)$ satisfying \eqref{eq:fixed-HJ}.

In light of Proposition~\ref{prop:linearized-optimality}, we find $(m,w)$ minimizes the
linearized problem with source $f$.    Then, we apply Proposition~\ref{prop:linear-duality}
to $(u,\lambda)$ and obtain  $w=-mD_pH(Du)=-m\gamma|Du|^{\gamma-2}Du.$  By 
substituting it into the constraint equation
$\Delta m-\nabla\cdot w=0$ yields $-\Delta m-\nabla\cdot\!\left(m\gamma|Du|^{\gamma-2}Du\right)=0$ in $\mathcal D'(\T^d)$.
Thus all equations and normalizations in Definition~\ref{def:solution} hold.

To complete the proof that \((u,\lambda,m,w)\) is a variational solution of \eqref{eq:mfg-system} as shown in Definition \ref{def:solution}, it remains to verify the integrability conditions.  Since
$H(p)=|p|^\gamma$ and $\gamma\qc=d(\gamma-1)$, we have
\begin{align}
 \|D_pH(Du)\|_{L^d}^d
 &=\gamma^d\int_{\T^d}|Du|^{d(\gamma-1)}\,d x
 \notag\\
 &=\gamma^d\int_{\T^d}\bigl(|Du|^\gamma\bigr)^{\qc}\,d x<\infty.
 \label{eq:feedback-Ld}
\end{align}
Moreover, Lemma~\ref{prop:translation} indicates $m\in L^{\alpha\qc}(\T^d)
 =L^{d/(d-2-\gp)}(\T^d).$
H\"older's inequality and \eqref{eq:feedback-Ld} therefore yield
\begin{equation}\label{eq:flux-improved}
 w=-mD_pH(Du)\in L^{r_*}(\T^d),
 \qquad
 \frac1{r_*}=\frac{d-2-\gp}{d}+\frac1d
 =\frac{d-1-\gp}{d}.
\end{equation}
  Uniqueness of $(m,w)$ was proved in
Lemma~\ref{prop:minimizer}.  Now, we have $(u,\lambda,m,w)$ is a variational solution as shown in Definition \ref{def:solution}.
\end{proof}

\section*{Acknowledgments}

 F.~Kong would like to thank Professor Alessandro Goffi for helpful discussions
and for bringing several relevant references to his attention.


\bibliographystyle{plain}
\bibliography{ref}
\end{document}